\documentclass[letterpaper]{article}
\usepackage[intlimits]{amsmath}
\usepackage{amsfonts, amssymb, amsthm}
\usepackage{epsfig}
\usepackage{color}

\newlength{\hchng}
\newlength{\vchng}
\newtheorem{thm}{Theorem}[section]
\newtheorem{prop}[thm]{Proposition}
\newtheorem{cor}[thm]{Corollary}
\newtheorem{ass}[thm]{Assumption}
\newtheorem{lemma}[thm]{Lemma}

\newtheorem{definition}[thm]{Definition}

\newtheorem{preremark}[thm]{Remark}
\newenvironment{remark}{\begin{preremark}\rm}{\medskip \end{preremark}}
\numberwithin{equation}{section}

\newcommand{\norm}[1]{\left\Vert#1\right\Vert}
\newcommand{\abs}[1]{\left\vert#1\right\vert}
\newcommand{\set}[1]{\left\{#1\right\}}
\newcommand{\R}{\mathbb R}
\newcommand{\eps}{\varepsilon}

\newcommand{\grad} {\nabla}
\newcommand{\lap} {\triangle}
\newcommand{\bdary} {\partial}
\newcommand{\dx} {\; \mathrm{d} x}
\newcommand{\dd} {\; \mathrm{d}}
\newcommand{\dist} {\mathrm{dist}}

\def\XXint#1#2#3{\quad {\setbox0=\hbox{$#1{#2#3}{\int}$}
       \vcenter{\hbox{$#2#3$}}\kern-0.5\wd0}}

\newcommand{\LI}{\mathcal{L}}
\newcommand{\MLp} {\mathrm{M}^+_\mathcal{L}}
\newcommand{\MLm} {\mathrm{M}^-_\mathcal{L}}

\newcommand{\si}{\delta}
\newcommand{\Mp} {\mathrm{M}^+}
\newcommand{\Mm} {\mathrm{M}^-}

\newcommand{\ro}{\rho_0}

\title{Regularity theory for fully nonlinear integro-differential equations}
\author{Luis Caffarelli and Luis Silvestre}

\begin{document}

\maketitle

\begin{abstract}
We consider nonlinear integro-differential equations, like the ones that arise from stochastic control problems with purely jump L\`evy processes. We obtain a nonlocal version of the ABP estimate, Harnack inequality, and interior $C^{1,\alpha}$ regularity for general fully nonlinear integro-differential equations. Our estimates remain uniform as the degree of the equation approaches two, so they can be seen as a natural extension of the regularity theory for elliptic partial differential equations.
\end{abstract}


\section{Introduction}
Integro-differential equations appear naturally when studying discontinuous stochastic processes. The generator of an $n$-dimensional L\`evy process is given by an operator with the general form
\begin{equation} \label{e:fulloperator}
Lu(x) = \sum_{ij} a_{ij} \partial_{ij} u + \sum_i b_i \partial_i u + \int_{\R^n} (u(x+y)- u(x) - \grad u(x) \cdot y) \ \chi_{B_1}(y) \dd \mu(y).
\end{equation}
The first term corresponds to the diffusion, the second to the drift, and the third to the jump part. In this paper we focus on the equations that we obtain when we consider purely jump processes; processes without diffusion or drift part. The operators have the general form
\begin{equation} \label{e:integral}
Lu(x) = \int_{\R^n} (u(x+y) - u(x) - \grad u(x) \cdot y \ \chi_{B_1}(y)) \dd \mu(y). 
\end{equation}
where $\mu$ is a measure such that $\int_{\R^n} \frac{|y|^2}{1+|y|^2} \dd \mu(y) < +\infty$.

The value of $L u(x)$ is well defined as long as $u$ is bounded in $\R^n$ and $C^{1,1}$ at $x$. These concepts will be made more precise later.

The operator $L$ described above is a linear integro-differential operator. In this paper we want to obtain results for nonlinear equations. We obtain this kind of equations in stochastic control problems \cite{So}. If in a stochastic game a player is allowed to choose from different strategies at every step in order to maximize the expected value of some function at the first exit point of a domain, a convex nonlinear equation emerges
\begin{equation} \label{e:mL}
Iu(x) = \sup_\alpha L_\alpha u(x)
\end{equation}

In a competitive game with two or more players, more complicated equations appear. We can obtain equations of the type
\begin{equation} \label{e:mmL}
Iu(x) = \inf_\beta \sup_\alpha L_{\alpha \beta} u(x) 
\end{equation}

The difference between \eqref{e:mmL} and \eqref{e:mL} is convexity. Alternatively, also an operator like $Iu(x) = \sup_\alpha \inf_\beta  L_{\alpha \beta} u(x)$ can be considered. A characteristic property of these operators is that
\begin{equation} \label{e:general}
\inf_{\alpha \beta} L_{\alpha \beta} v(x) \leq I (u+v) (x) - Iu (x) \leq \sup_{\alpha \beta} L_{\alpha \beta} v(x)
\end{equation}

A more general and better description of the nonlinear operators we want to deal with is the operators $I$ for which \eqref{e:general} holds for some family of linear integro-differential operators $L_{\alpha \beta}$. The idea is that an estimate on $I(u+v) - Iu$ by a suitable extremal operator can be a replacement for the concept of ellipticity. Indeed, if we consider the extremal Pucci operators \cite{CC}, $M^+_{\lambda,\Lambda}$ and $M^-_{\lambda,\Lambda}$, and we have $M^-_{\lambda,\Lambda} v(x) \leq  I(u+v) - Iu \leq M^+_{\lambda,\Lambda} v(x)$, then it is easy to see that $I$ must be an elliptic second order differential operator. If instead we compare with suitable nonlocal extremal operators, we will have a concept of ellipticity for nonlocal equations. We will give a precise definition in section \ref{s:maximal} (Definition \ref{d:axiomatic}).

We now explain the natural Dirichlet problem for a nonlocal operator. Let $\Omega$ be an open domain in $\R^n$. We are given a function $g$ defined in $\R^n \setminus \Omega$, which is the boundary condition. We look for a function $u$ such that
\begin{align*}
Iu(x) &= 0 && \text{for every } x \in \Omega \\
u(x) &= g(x) &&\text{for } x \in \R^n \setminus \Omega
\end{align*}

Notice that the boundary condition is given in the whole complement of $\Omega$ and not only $\bdary \Omega$. This is because of the nonlocal character of the operator $I$. From the stochastic point of view, it corresponds to the fact that a discontinuous L\`evy process can exit the domain $\Omega$ for the first time jumping to any point in $\R^n \setminus \Omega$.

In this paper we will focus mainly in the regularity properties of solutions to an equation $Iu=0$. We will briefly present a very general comparison principle from which existence of solutions can be obtained in smooth domains. In order to obtain regularity results, we must assume some \emph{nice} behavior of the measures $\mu$. Basically, our assumption is that they are symmetric, absolutely continuous and not too degenerate. To fix ideas, we can think of integro-differential operators with a kernel comparable with the respective kernel of the fractional laplacian $-(-\lap)^{\sigma/2}$. In this respect, the theory we develop can be understood as a theory of viscosity solutions for fully nonlinear equations of fractional order.

In this paper we would like to quickly present the necessary definitions and then prove some regularity estimates. Our results in this paper are
\begin{itemize}
\item A comparison principle for a general nonlinear integro-differential equation.
\item A nonlocal version of the Alexandroff-Backelman-Pucci estimate.
\item The Harnack inequality for integro-differential equations with kernels that are comparable with the ones of the fractional laplacian but can be very discontinuous.
\item A H\"older regularity result for the same class of equations as the Harnack inequality.
\item A $C^{1,\alpha}$ regularity result for a large class of nonlinear integro-differential equations.
\end{itemize}

Even though there are some known results about Harnack inequalities and H\"older estimates for integro-differential equations with either analytical proofs \cite{S1} or probabilistic proofs \cite{BK}, \cite{BK2}, \cite{BL}, \cite{SV}, the estimates in all these previous results blow up as the order of the equation approaches $2$. In this way, they do not generalize to elliptic differential equations. We provide estimates that remain uniform in the degree and therefore make the theory of integro-differential equations and elliptic differential equations appear somewhat unified. Consequently, our proofs are more involved than the ones in the bibliography.

In this paper we only consider nonlinear operators that are translation invariant. The \emph{variable coefficient} case will be considered in future work. I future papers, we are also planning to address the problem of the interior regularity of the integro-differential Hamilton-Jacobi-Bellman equation. This refers to the equation involving a convex nonlocal operator like \eqref{e:mL}. In that case we obtain an analogue of the Evans-Krylov theorem proving that the solutions to the equation have enough regularity to be classical solutions.

The structure of the paper is as follows. After this introduction, the second section presents the appropriate definitions of subsolution and supersolution of an integro-differential equation in the viscosity sense. In our definition we allow any kind of discontinuities outside of the domain of the equations. In the third section we give the general description of the elliptic nonlocal equations that we want to study. We define a nonlocal elliptic operator by comparing its increments with a suitable maximal operator. This definition is more general than \eqref{e:mmL}. In the fourth section we study the stability of our definitions. A comparison principle is proven in section five under very mild assumptions. Next, in section six we show how to obtain an elliptic partial differential equation as a limit of integro-differential equations. We believe one of the most nontrivial results in the paper is the nonlocal ABP estimate developed in section seven. In sections eight and nine we construct a special function and prove some pointwise estimates that will help in proving the Harnack inequality and H\"older estimates in sections ten and eleven. In section twelve we show the $C^{1,\alpha}$ estimates. And finally in section thirteen we show how to generalize our previous results when our operators have truncated kernels. This last section is important for applications since very often the kernels of an integro differential equation are comparable to the ones of the fractional laplacian only in a neighborhood of the origin.

\section{Definitions}
\label{s:definitions}

As we mention in the introduction, equation \eqref{e:integral} was given in too much generality for our purposes. We will restrict our attention to the operators where $\mu$ is given by a symmetric kernel $K$. It takes the form
\begin{equation}\label{e:linear0}
Lu(x) = \mathrm{PV} \int_{\R^n} (u(x+y) - u(x)) K(y) \dd y \ .
\end{equation}

The kernel $K$ must be a positive function, satisfy $K(y) = K(-y)$, and also
\begin{equation} \label{e:minimumassumptionforlinear}
 \int_{\R^n} \frac{|y|^2}{|y|^2+1} K(y) \dd y < +\infty
\end{equation}
It is not necessary to subtract the term $-\grad u(x) \cdot y \chi_{B_1}$ if we think of the integral in the principal value sense. Alternatively, due to the symmetry of the kernel $K$, the operator can also be written as
\[ Lu(x) = \frac{1}{2} \int_{\R^n} (u(x+y) + u(x-y) - 2u(x)) K(y) \dd y \ . \]

In order to simplify the notation, we will write $\si(u,x,y) := u(x+y) + u(x-y) - 2u(x)$. The expression for $L$ can be written shortly as
\begin{equation} \label{e:linear}
 Lu(x) = \int_{\R^n} \si(u,x,y) K(y) \dd y \ .
\end{equation}
for some kernel $K$ (which would be half of the one of \eqref{e:linear0}). We will alternate from writing the operators in the form \eqref{e:linear0} or \eqref{e:linear} whenever it is convenient.

The nonlinear integro-differential operators that arise in stochastic control have the form \eqref{e:mmL}
where we think that for each $L_{\alpha \beta}$ we have a kernel $K_{\alpha \beta}$ so that $L_{\alpha \beta}$ has the form \eqref{e:linear}. We will define a more general form for nonlinear integro-differential operators in section \ref{s:maximal}. 

The minimum assumption in order to have $Iu$ well defined is that every kernel $K_{\alpha \beta}$ must satisfy \eqref{e:minimumassumptionforlinear} in a uniform way. More precisely
\begin{equation} \label{e:minimumassumption}
 \text{if } K(y) := \sup_{\alpha \beta} K_{\alpha \beta} (y) \ \text{  then  } \int_{\R^n} \frac{|y|^2}{|y|^2+1} \ K (y) \ \dd y < +\infty 
\end{equation}


The value of $Iu$ can be evaluated in the classical sense if $u \in C^{1,1}$. If we want to evaluate the value of $Iu (x)$ at one point $x$ only, we need $u$ to be punctually $C^{1,1}$ in the sense of the following definition.

\begin{definition}
A function $\varphi$ is said to be $C^{1,1}$ at the point $x$, and we write $u \in C^{1,1}(x)$, if there is a vector $v \in \R^n$ and a number $M>0$ such that 
\[ |\varphi(x+y) - \varphi(x) - v \cdot y| \leq M |y|^2 \qquad \text{for $|y|$ small enough.}\]
\end{definition}

We give a definition of viscosity sub- and super-solutions for integro-differential equations by evaluating the operators in $C^{1,1}$ test functions that \emph{touch} the function $u$ from either above or below. Often for nonlocal equations the definition is given by test functions that remain on one side of $u$ in the whole space $\R^n$. We take a sligtly different approach. We consider a test function $\varphi$ that touches $u$ at a point $x$ and remains on one side of $u$ but it is only defined locally, in a neighborhood $N$ of $x$. Then we complete $\varphi$ with the tail of $u$ to evaluate the integrals \eqref{e:linear}. We do this in order to allow arbitrary discontiuities in the function $u$ outside of the domain $\Omega$ where it may be a solution of the equation.

\begin{definition} \label{d:viscositysolutions}
A function $u :\R^n \to \R$, upper (lower) semi continuous in $\overline \Omega$, is said to be a subsolution (supersolution) to $Iu = f$, and we write $Iu\geq f$ ($Iu \leq f$), if every time all the following happen
\begin{itemize}
\item $x$ is any point in $\Omega$.
\item $N$ is a neighborhood of $x$ in $\Omega$.
\item $\varphi$ is some $C^2$ function in $\overline N$.
\item $\varphi(x) = u(x)$.
\item $\varphi(y) > u(y)$ ($\varphi(y) < u(y)$) for every $y \in N \setminus \{x\}$.
\end{itemize}
Then if we let \[ v := \begin{cases}
               \varphi &\text{in } N \\
	       u &\text{in } \R^n \setminus N \ ,
              \end{cases} \]
we have $Iv(x) \geq f(x)$ ($Iv(x) \leq f(x)$).

A solution is a function $u$ which is both a subsolution and a supersolution.
\end{definition}

Note that Definition \ref{d:viscositysolutions} is essentially the same as Definition 2 in \cite{BI}.

For the set of test functions, we could also use a function $\varphi$ that is $C^{1,1}$ only at the contact point $x$. This is a larger set of test functions, so a priori it may provide a stronger concept of solution. In section \ref{s:stability} we will show that the two approaches are actually equivalent.

Usually the nonlocal operators $I$ allow some growth at infinity. If the value of $Iu(x)$ is well defined every time $u \in C^{1,1}(x)$ and $u \in L^1\left(\R^n, w \right)$ for some weight $w$ that is locally bounded, then the above definition would apply for semicontinuous functions in $\overline \Omega$ that are in $L^1(\R^n, w)$ but not necesarily bounded. In most cases, our regularity results in this paper can be extended to the unbounded case by truncating the function and adding an error term in the right hand side.

\section{Maximal operators}
\label{s:maximal}

In \eqref{e:mL} and \eqref{e:mmL} we consider the supremum or an \emph{inf-sup} of a collection of linear operators. Let us consider a collection of linear operators $\LI$ that includes all of them. The maximal and a minimal operator respect to $\LI$ are defined as:
\begin{align}
\MLp v(x) &= \sup_{L \in \LI} L u(x) \\
\MLm v(x) &= \inf_{L \in \LI} L u(x) .
\end{align}

For example, an important class that we will use for regularity results is given by the class $\LI_0$ of operators $L$ of the form \ref{e:linear} with 
\begin{equation} \label{e:uniformellipticity} 
 (2 - \sigma) \frac{\lambda}{|y|^{n+\sigma}}\leq K(y) \leq (2 - \sigma) \frac{\Lambda}{|y|^{n+\sigma}} \, ,
\end{equation}
then $\Mp_{\LI_0}$ and $\Mm_{\LI_0}$ take a very simple form: 
\begin{align}
\Mp_{\LI_0} v(x) &= (2 - \sigma) \int_{\R^n} \frac{\Lambda \si(v,x,y)^+ - \lambda \si(v,x,y)^-}{|y|^{n+\sigma}} \dd y \label{e:Mp}\\
\Mm_{\LI_0} v(x) &= (2 - \sigma) \int_{\R^n} \frac{\lambda \si(v,x,y)^+ - \Lambda \si(v,x,y)^-}{|y|^{n+\sigma}} \dd y \ . \label{e:Mm} 
\end{align}
We will use these maximal operators to obtain regularity estimates. The factor $(2-\sigma)$ is important when $\sigma \to 2$. We need such factor if we want to obtain second order differential equations as limits of integro-differential equations. In terms of the regularity, we need the factor $(2-\sigma)$ for the estimates not to blow up as $\sigma \to 2$.


Another interesting class is given when the kernels have the form
\[ K(y) = (2-\sigma) \frac{ y^t A y}{|y|^{n+2+\sigma}} \ , \]
for symmetric matrices $A$ such that $\lambda I \leq  A \leq \Lambda I$. 
This is a smaller class than the $\LI_0$ above if we choose the respective constants $\lambda$ and $\Lambda$ accordingly, but it is all we need to recover the classical Pucci extremal operators \cite{CC} as $\sigma \to 2$.

Let $K(x)$ be the suppremum of $K_\alpha(x)$ where $K_\alpha$ are all the kernels of all operators $L \in \LI$. As a replacement for \eqref{e:minimumassumption}, for any class $\LI$ we will assume
\begin{equation} \label{e:minimumassumptionforclass}
\int_{\R^n} \frac{|y|^2}{|y|^2+1} \ K (y) \ \dd y < +\infty 
\end{equation}

Using the extremal operators, we give a general definition of ellipticity for nonlocal equations. The following is the kind of operators for which the results in this paper apply.

\begin{definition} \label{d:axiomatic}
Let $\LI$ be a class of linear integro differential operators. We always assume \eqref{e:minimumassumptionforclass}. An elliptic operator $I$ respect to $\LI$ is an operator with the following properties:
\begin{itemize}
\item If $u$ is any bounded function, $Iu(x)$ is well defined every time $u \in C^{1,1}(x)$.
\item If $u$ is $C^2$ in some open set $\Omega$, then $Iu(x)$ is a continuous function in $\Omega$.
\item If $u$ and $v$ are bounded functions $C^{1,1}(x)$, then
\begin{equation} \label{eq:puccicontinuous}
\MLm (u-v)(x) \leq Iu(x) - Iv(x) \leq \MLp(u-v)(x)
\end{equation}
\end{itemize}
\end{definition}

Definition \ref{d:viscositysolutions} applies for the general nonlocal elliptic operators of Definition \ref{d:axiomatic} \emph{mutatis mutandis}.

Definition \ref{d:axiomatic} may apply to operators $I$ whether or not they are translation invariant. However, in this paper we will only focus on the translation invariant case. In other words, for all nonlinear operators $I$ in this paper we assume that $\tau_z I u = I (\tau_z u)$, where $\tau_z$ is the translation operator $\tau_z u(x) := u(x-z)$.

We will show that any operator of the form \eqref{e:mmL} is elliptic with respect to any class that contains all the operators $L_{\alpha \beta}$ as long as condition \eqref{e:minimumassumption} is satisfied (Lemma \ref{l:puccicontinuous} and Lemma \ref{l:c11}). However the Definition \ref{d:axiomatic} allows a richer class of equations. For example we can consider an operator $I$ given by
\[ I u (x) = \int_{\R^n} \frac{G(u(x+y) - u(x))}{|y|^{n+\sigma}} \dd \sigma \]
for any monotone Lipschitz function $G$ such that $G(0)=0$. This operator $I$ would be elliptic with respect to the class $\LI_0$.

\begin{lemma} \label{l:puccicontinuous}
Let $I$ be an operator like in \eqref{e:mmL} and $\LI$ be any collection of integro-differential operators. Assume every $L_{\alpha \beta}$ belongs to the class $\LI$. Then for every $u,v \in C^{1,1}(x)$ we have
\[ \MLm (u-v) (x) \leq Iu(x) - Iu(x) \leq \MLp (u-v)(x) \]
\end{lemma}

\begin{proof}
Since $u \in C^{1,1}(x)$, $L_{\alpha \beta} u(x)$ is defined classically for any $L_{\alpha \beta}$. Let's assume first that $I$ is convex. We have
\[ Iu(x) = \sup_\alpha L_\alpha u(x) \ . \]
Thus, for every $\eps>0$, there is an $\alpha_1$ and an $\alpha_2$ such that
\begin{align*}
Iu(x) - L_{\alpha_1} u(x) &< \eps\\
Iv(x) - L_{\alpha_2} v(x) &< \eps.
\end{align*}
Thus we have
\begin{align}
L_{\alpha_2} u(x) - L_{\alpha_2} v(x) - \eps &\leq Iu(x) - Iv(x) \leq L_{\alpha_1} u(x) - L_{\alpha_1} v(x) + \eps \\
\MLm (u-v)(x) - \eps &\leq Iu(x) - Iv(x)\leq \MLp (u-v) (x) + \eps \ .
\end{align}

Since we can take $\eps$ as small as we want, we obtain $\MLm (u-v)(x) \leq Iu(x) - Iv(x)\leq \MLp (u-v) (x)$ for every convex $I$.

For the nonconvex case, we can write
\[ Iu(x) = \inf_\beta \sup_\alpha L_{\alpha \beta} u(x) = \inf_\beta I_\beta u(x), \]
where the $I_\beta$ is the convex operator given by $I_\beta u(x) = \sup_\alpha L_{\alpha \beta} u(x)$. Now a similar idea applies.

For every $\eps>0$, there is an $\beta_1$ and an $\beta_2$ such that
\begin{align*}
Iu(x) - I_{\beta_1} u(x) &< \eps\\
Iv(x) - I_{\beta_2} v(x) &< \eps.
\end{align*}
Thus we have
\begin{align}
I_{\beta_1} u(x) - I_{\beta_1} v(x) - \eps &\leq Iu(x) - Iv(x) \leq I_{\beta_1} u(x) - I_{\beta_1} v(x) + \eps \\
\MLm (u-v)(x) - \eps &\leq Iu(x) - Iv(x)\leq \MLp (u-v) (x) + \eps \ .
\end{align}

Taking $\eps \to 0$, we obtain $\MLm (u-v)(x) \leq Iu(x) - Iv(x)\leq \MLp (u-v) (x)$ for any $I$ of the form \eqref{e:mmL}.
\end{proof}

The family of operators that satisfy the condition \eqref{e:uniformellipticity} have another very curious property. Definition \ref{d:viscositysolutions} is made so that we never have to evaluate the operator $I$ in the original function $u$. Every time we touch $u$ with a smooth function $\varphi$ from above, we construct a test function $v \in C^{1,1}(x)$ to evaluate $I$. It is somewhat surprising that if $I$ is any nonlinear operator $I$ that is an $\inf \sup$ (or a $\sup \inf$) of linear operators that satisfy \eqref{e:uniformellipticity}, then this turn out to be unnecessary, since $I$ can be evaluated classically in $u$ at those points $x$ where $u$ can be touched by above with a paraboloid. This is explained in the next lemma.

\begin{lemma} \label{l:classic2}
Let $I$ be an operator like in \eqref{e:mmL} so that for every $K_{\alpha \beta}$ the equation \eqref{e:uniformellipticity} holds.

If we have a subsolution, $Iu \geq f$ in $\Omega$ and $\varphi$ is a $C^2$ function that touches $u$ from above at a point $x \in \Omega$, then $I u(x)$ is defined in the classical sense and $I u(x) \geq f(x)$.
\end{lemma}

\begin{proof}
For any $r>0$, we define
\[
 v_r = \begin{cases}
        \varphi & \text{in } B_r \\
        u & \text{in } \R^n \setminus B_r \ ,
       \end{cases}
\]
and we have $\Mp v_r(x) \geq I v_r (x) \geq f(x)$. Thus
\[ (2-\sigma) \int \si(v_r,x,y)^+ \frac{\Lambda}{|y|^{n+\sigma}} - \si(v_r,x,y)^- \frac{\lambda}{|y|^{n+\sigma}} \dd y \geq f(x) \]

Since $\varphi$ touches $u$ from above at $x$, for any $y \in \R^n$, $\si(v_r,x,y) \geq \si(u,x,y)$. Since $v_r \in C^{1,1}(x)$, $|\si(v_r,x,y)|/|y|^{n+\sigma}$ is integrable, and then so is $\si(u,x,y)^+ /|y|^{n+\sigma}$. 

We have
\[ (2-\sigma) \int \si(v_r,x,y)^- \frac{\lambda}{|y|^{n+\sigma}} \dd y \leq (2-\sigma) \int \si(v_r,x,y)^+ \frac{\Lambda}{|y|^{n+\sigma}} \dd y - f(x) \]

Since $\varphi$ touches $u$ from above at $x$, $\si(v_r,x,y)$ will decrease as $r$ decreases. Therefore, for every $r < r_0$
\begin{equation} \label{e:ii22}
(2-\sigma) \int_{\R^n} \si(v_r,x,y)^- \frac{\lambda}{|y|^{n+\sigma}} \dd y \leq (2-\sigma) \int_{\R^n} \si(v_{r_0},x,y)^+ \frac{\Lambda}{|y|^{n+\sigma}} \dd y - f(x) 
\end{equation}

But $\si(v_r,x,y)^-$ is monotone increasing as $r$ decreases, and it converges to $\si(u,x,y)^-$ as $r \to 0$. From monotone convergence theorem
\[\lim_{r \to 0} (2-\sigma) \int_{\R^n} \si(v_r,x,y)^- \frac{\lambda}{|y|^{n+\sigma}} \dd y = (2-\sigma) \int_{\R^n} \si(u,x,y)^- \frac{\lambda}{|y|^{n+\sigma}} \dd y \ .
\]
And from \eqref{e:ii22}, the integrals are uniformly bounded and thus
\[ (2-\sigma) \int_{\R^n} \si(u,x,y)^- \frac{\lambda}{|y|^{n+\sigma}} \dd y \leq (2-\sigma) \int_{\R^n} \si(v_{r_0},x,y)^+ \frac{\Lambda}{|y|^{n+\sigma}} \dd y - f(x) < +\infty \]

Therefore, $\si(u,x,y)/|y|^{n+\sigma}$ is integrable, and $L_{\alpha \beta} u$ is well defined in the classical sense for any $\alpha$ and $\beta$. Thus, $Iu(x)$ is computable in the classical sense. The difference $\si(v_r-u,x,y)/|y|^{n+\sigma}$ is monotone decreasing as $r \searrow 0$, converges to zero, and it is bounded by the integrable function $\si(v_{r_0}-u,x,y)/|y|^{n+\sigma}$. We can pass to the limit in the following expression:
\begin{align}
 \lim_{r \to 0} \Mp (v_r - u) (x) &= \lim_{r \to 0} (2-\sigma) \int \si(v_r - u,x,y)^+ \frac{\Lambda}{|y|^{n+\sigma}} \dd y \\
&= 0
\end{align}

Now we use Lemma \ref{l:puccicontinuous} to conclude
\[
I u(x) \geq I v_r (x) + \Mm (u-v_r) = f(x) - \Mp(v_r - u) \to f(x)
\]

So $I u(x) \geq f(x)$.
\end{proof}

Lemma \ref{l:classic2} is convenient to make proofs involving $\Mp$ and $\Mm$ because it allows to deal with viscosity solutions almost as if they were classical solutions. It is not clear to what other types of nonlinear operators a result like Lemma \ref{l:classic2} would extend.

\section{Stability properties}
\label{s:stability}

In this section we show a few technichal properties of the operators $I$ like \eqref{e:mmL}. First that if $u \in C^{1,1}(\Omega)$ then $Iu$ is continuous in $\Omega$. As it was mentioned in the previous sections, it is necessary to justify that the operators of the form \eqref{e:mmL} satisfy the conditions of Definition \ref{d:axiomatic}. Next, we will show that our notion of viscosity solutions allows to touch with solutions that are only punctually $C^{1,1}$ instead of $C^2$ in a neighborhood of the point. Then we will show the important stability property of Definition \ref{d:viscositysolutions}. Namely we show that if a sequence of subsolutions (or supersolutions) in $\Omega$ converges in a suitable way on any compact set in $\R^n$, then the limit is also a subsolution (or supersolution).

We start with a technichal lemma.

\begin{lemma} \label{l:realanalysis}
Let $f \in L^\infty(\R^n)$ and $g_\alpha$ be a family of functions so that $|g_\alpha(x)| \leq g(x)$ for some $L^1$ function $g$. Then the family $f \ast g_\alpha$ is equicontinuous in every compact set.
\end{lemma}

\begin{proof}
Let $K$ be a compact set in $\R^n$. Let $\eps>0$. Since $g \in L^1$, we can pick a large $R$ so that $K \subset B_R$ and
\[ \norm{f}_{L^\infty} \left( \int_{\R^n \setminus B_R(x)} g(y) \dd y \right) \leq \eps/8 \]
for any $x \in K$. We write $f = f_1 + f_2$, where $f_1 = f \chi_{B_{2R}}$ and $f_2 = f \chi_{\R^n \setminus B_{2R}}$. From the above inequality, we have $|f_2 \ast g_\alpha| \leq \eps/8$ in $K$.

Since $g \in L^1$, there is a $\delta_0 > 0$ so that
\begin{equation} \label{e:s1}
\int_A g(x) \dd x < \frac{\eps}{16 \norm{f}_{L^\infty}} \qquad \text{ for any set } |A| < \delta_0
\end{equation}

Let $\eta_t$ be a standard mollifier with compact support. We have $f_1 \ast \eta_t \to f_1$ a.e. (in every Lebesgue point of $f_1$). Recall that the support of $f_1$ is in $B_R$. For $t$ large, $f_1 \ast \eta_t=0$ ouside $B_{4R}$. By Egorov's theorem, there is a set $A \subset B_{4R}$ such that
\begin{align}
&|A| < \delta_0 \label{e:s11} \\
&f_1 \ast \eta_t \to f_1 \qquad \text{uniformly in } \R^n \setminus A
\end{align}

In particular, there is a $\tilde f_1 = f_1 \ast \eta_{t_0}$ such that $|f_1 - \tilde f_1| < \frac{\eps} { 8 \norm{g}_{L^1}}$ in $\R^n \setminus A$. We have
\begin{equation}
\norm{(f_1-\tilde f_1) (1-\chi_A) \ast g_\alpha}_{L^\infty} \leq \norm{(f_1-\tilde f_1) (1-\chi_A)}_{L^\infty} \norm{g_\alpha}_{L^1} < \frac{\eps}{8} 
\end{equation}

On the other hand, from \eqref{e:s1} and \eqref{e:s11}, we also get
\begin{equation} \label{e:s2}
\norm{(f_1-\tilde f_1) \chi_A \ast g_\alpha}_{L^\infty} < \frac{\eps}{8}
\end{equation}

Since $\tilde f_1$ is continuous and $\norm{g_\alpha}_{L^1}$ is bounded, the family $\tilde f_1 \ast g_\alpha$ is equicontinuous. There is a $\delta>0$ so that $|\tilde f_1 \ast g_\alpha(x) - \tilde f_1 \ast g_\alpha(y)| < \eps/4$ every time $|x-y| < \delta$. Moreover
\begin{align*}
 |f \ast g_\alpha(x) - f \ast g_\alpha(y)| &\leq |\tilde f_1 \ast g_\alpha(x) - \tilde f_1 \ast g_\alpha(y)| + |(f_1-\tilde f_1) \ast g_\alpha(x) - (f_1-\tilde f_1) \ast g_\alpha(y)| \\
&\qquad + |f_2 \ast g_\alpha(x) - f_2 \ast g_\alpha(y)| \\
&\leq \eps/4 + |(f_1-\tilde f_1)\chi_A \ast g_\alpha(x)| + |(f_1-\tilde f_1)\chi_A \ast g_\alpha(y)| \\
&\phantom{ \leq \eps/4 + } + |(f_1-\tilde f_1)(1-\chi_A) \ast g_\alpha(x)| + |(f_1-\tilde f_1)(1-\chi_A) \ast g_\alpha(y)| \\
&\phantom{ \leq \eps/4 + } + |f_2 \ast g_\alpha(x)| + |f_2 \ast g_\alpha(y)|\\
&\leq \eps
\end{align*}
for any $\alpha$ and every time $|x-y| < \delta$.
\end{proof}

\begin{lemma} \label{l:c11}
Let $I$ be an operator like in \eqref{e:mmL}, assuming only \eqref{e:minimumassumption}. Let $v$ be a bounded function in $\R^n$ and $C^{1,1}$ in some set $\Omega$. Then $Iv$ is continuous in $\Omega$.
\end{lemma}

\begin{proof}
We must prove the $L_{\alpha \beta} v$ in \eqref{e:mmL} are equicontinuous. Like in \eqref{e:minimumassumption}, we write $K = \sup_{\alpha \beta} K_{\alpha \beta}$.

Let $\eps>0$ and $x_0 \in \Omega$. Since $v$ is $C^{1,1}$ in $\Omega$, there is a constant $C$ so that
\[ |\si(v,x,y)| < C|y|^2 \qquad \text{if } x \in \Omega \text{ and } |y| < \dist(x, \bdary \Omega) \]

Let $r>0$ such that 
\[ \int_{B_r} C|y|^2 K(y) \dd y < \eps/3 \]

We have
\begin{align*}
 L_{\alpha \beta} v(x) &= \int_{\R^n} \si(v,x,y) K_{\alpha \beta} (y) \dd y \\
&= \int_{B_r} \si(v,x,y) K_{\alpha \beta} (y) \dd y + \int_{\R^n \setminus B_r} \si(v,x,y) K_{\alpha \beta} (y) \dd y \\
&=: w_1(x) + w_2(x)
\end{align*}
where 
\[ |w_1| = \abs{\int_{B_r} \si(v,x,y) K_{\alpha \beta} (y) \dd y} \leq \int_{B_r} C|y|^2 K(y) \dd y < \eps/3 \] 
and
\begin{align*}
w_2 &= \int_{\R^n \setminus B_r} (v(x+y) + v(x-y) - 2v(x)) K_{\alpha \beta} (y) \dd y \\
&= v \ast g_{\alpha \beta} + v \ast \hat g_{\alpha \beta} - 2\left( \int g_{\alpha \beta} \dd y \right) v
\end{align*}
where $g_{\alpha \beta}(y) = \chi_{\R^n \setminus B_r}(y) K_{\alpha \beta}(y)$ and $\hat g_{\alpha \beta}(y) = g_{\alpha \beta}(-y)$. For any $\alpha$ and $\beta$, $g_{\alpha \beta} \leq \chi_{\R^n \setminus B_r} K$, which is in $L^1$.  From Lemma \ref{l:realanalysis}, $w_2$ is equicontinuous. So there is a $\delta>0$ such that
\[ |w_2(x) - w_2(x_0)| < \eps/3 \qquad \text{if } |x-x_0|<\delta \]

Therefore
\[ L_{\alpha \beta} v(x) - L_{\alpha \beta} v(x_0)| \leq |w_1(x)|+|w_1(x_0)|+|w_2(x)-w_2(x_0)| < \eps \]
uniformly in $\alpha$ and $\beta$. Thus $|Iv(x) - Iv(x_0)| < \eps$ every time $|x-x_0|<\delta$.
\end{proof}

When we gave the definition of viscosity solutions in section \ref{s:definitions}, we used $C^2$ test functions. Now we show that it is equivalent to use punctually $C^{1,1}$ functions.

\begin{lemma} \label{l:tc11}
Let $I$ be elliptic respect to some class $\LI$ in the sense of Definition \ref{d:axiomatic}. Let $u : \R^n \to \R$ be an upper semicontinuous function such that $I u \geq 0$ in $\Omega$ in the viscosity sense. Let $\varphi : \R^n \to \R$ be a bounded function, punctually $C^{1,1}$ at a point $x \in \Omega$. Assume $\varphi$ touches $u$ from above at $x$. Then $I\varphi (x)$ is defined in the classical sense and $I \varphi(x) \geq f(x)$.
\end{lemma}

\begin{proof}
Since $\varphi$ is $C^{1,1}$, the expression \eqref{e:linear} is clearly integrable for every $\alpha$ and $\beta$ and $I\varphi(x)$ is defined classically.

Also because $\varphi$ is $C^{1,1}$, there is a quadratic polynomial $q$ touching $\varphi$ from above at $x$. Let
\[ v_r(x) = \begin{cases}
             q & \text{in } B_r \\
	     u & \text{in } \R^n \setminus B_r \ .
            \end{cases} \]

Since $I u \geq f$ in $\Omega$ in the viscosity sense then $I v_r (x) \geq f(x)$ with $I v_r(x)$ well defined. Moreover let
\[ u_r(x) = \begin{cases}
             q & \text{in } B_r \\
	     \varphi & \text{in } \R^n \setminus B_r \ .
            \end{cases} \]
we have
\begin{align*}
I \varphi(x) &\geq I u_r(x) + \MLm (\varphi-u_r)(x) \geq I u_r(x) && \text{since $\varphi-u_r$ has a minimum at $x$}\\
& \geq I v_r(x) + \MLm (u_r - v_r) (x) \\
& \geq f(x) + \MLm (u_r - v_r) (x) \\
& \geq f(x) + \int_{B_r} \si(q-\varphi,x,y)^- K(y) \dd y && \text{where $K$ is the one from \eqref{e:minimumassumptionforclass}} \\
& \geq f(x) - C \int_{B_r} |y|^2 K(y) \dd y \ .
\end{align*}

Since $|y|^2 K(y)$ is integrable in a neighborhood of the origin, the expression
\[ \int_{B_r} C |y|^2 K(y) \dd y \]
goes to zero as $r \to 0$. Thus, for any $\eps>0$, we can find a small $r$ so that
\[ I \varphi(x) \geq f(x) - \eps \ . \]
Therefore $I \varphi(x) \geq f(x)$.
\end{proof}

One of the most useful properties of viscosity solutions is their stability under uniform limits on compact sets. We will prove a slightly stronger result. We show that the notion of viscosity supersolution is stable with respect to the natural limits for lower semicontinuous functions. This type of limit is well known and usually called $\Gamma$-limit.

\begin{definition}[$\Gamma$-convergence]
A sequence of lower semicontinuous functions $u_k$ $\Gamma$-converges to $u$ in a set $\Omega$ if the two following conditions hold
\begin{itemize}
\item For every sequence $x_k \to x$ in $\Omega$, $\liminf_{k \to \infty} u_k(x_k) \geq u(x)$.
\item For every $x \in \Omega$, there is a sequence $x_k \to x$ in $\Omega$ such that $\limsup_{k \to \infty} u_k(x_k) = u(x)$.
\end{itemize}
\end{definition}

Naturally, a uniformly convergent sequence $u_k$ would also converge in the $\Gamma$ sense. An important property of $\Gamma$-limits is that if $u_k$ $\Gamma$-converges to $u$, and $u$ has a strict local minimum at $x$, then $u_k$ will have a local minimum at $x_k$ for a sequence $x_k \to x$.

\begin{lemma} \label{l:stability}
Let $I$ be elliptic in the sense of Definition \ref{d:axiomatic} and $u_k$ be a sequence of functions that are bounded in $\R^n$ and lower semicontinuous in $\Omega$ such that
\begin{itemize}
 \item $I u_k \leq f_k$ in $\Omega$.
\item $u_k \to u$ in the $\Gamma$ sense in $\Omega$.
\item $u_k \to u$ a.e. in $\R^n$.
\item $f_k \to f$ locally uniformly in $\Omega$.
\end{itemize}

Then $Iu \leq f$ in $\Omega$.
\end{lemma}

\begin{proof}
Let $\varphi$ be a test function from below for $u$ touching at a point $x$ in a neighborhood $N$.

Since $u_k$ $\Gamma$-converges to $u$ in $\Omega$, for large $n$, we can find $x_k$ and $\delta_k$ such that $\varphi+d_k$ touches $u_k$ at $x_k$. Moreover $x_k \to x$ and $d_k \to 0$ as $k \to + \infty$.

Since $I u_k \leq f_k$, if we let
\[ v_k = \begin{cases}
          \varphi + d_k & \text{in } N \\
          u_k & \text{in } \R^n \setminus N \ ,
         \end{cases} \]
we have $I v_k (x_k) \leq f_k(x_k)$.

Let $z \in N$ be such that $\dist(z,\bdary N) > \rho >0$. We have
\begin{align*}
\abs{I v_k (z) - I v(z)} &\leq \max\left( |\MLp (v_k-v)(z)|, |\MLp (v-v_k)(z)| \right) \\
&\leq \sup_{L \in \LI} \abs{ L (v_k - v)(z) }  \\
&\leq \int_{\R^n} |\si(v_k-v,z,y)| K(y) \dd y \\
&\leq \int_{\R^n \setminus B_\rho} |\si(v_k-v,z,y)| K(y) \dd y
\end{align*}

The sequence $v_k$ is bounded and $\si(v_k-v,z,y)$ converges to zero almost everywhere. Since $K \in L^1(\R^n \setminus B_\rho)$, we can use dominated convergence theorem to show that the above expression goes to zero as $k \to +\infty$. Moreover the convergence is uniform in $z$. We obtain $I v_k \to I v$ locally uniformly in $N$.

From Definition \ref{d:axiomatic}, we have that $Iv$ is continuous in $N$. We now compute
\[ |Iv_k(x_k) - I v(x)| \leq |Iv_k(x_k) - Iv(x_k)| + |Iv(x_k) - Iv(x)| \to 0 \ . \]

So $Iv_k(x_k)$ converges to $Iv(x)$, as $k \to +\infty$. Since $x_k \to x$ and $f_k \to f$ locally uniformly, we also have $f_k(x_k) \to f(x)$, which finally implies $Iv(x) \leq f(x)$.
\end{proof}

In the previous lemma we showed the stability of supersolutions under $\Gamma$ limits. Naturally, we also have the corresponding result for subsolutions. In that case we would consider the natural limit in the space of upper semicontinuous functions which is the same as the $\Gamma$-convergence of $-u_k$ to $-u$. As a corollary, we obtain the stability under uniform limits. 

\begin{cor}
Let $I$ be elliptic in the sense of Definition \ref{d:axiomatic} and $u_k$ be a sequence of functions that are bounded in $\R^n$ and continuous in $\Omega$ such that
\begin{itemize}
 \item $I u_k = f_k$ in $\Omega$.
\item $u_k \to u$ locally uniformly in $\Omega$.
\item $u_k \to u$ a.e. in $\R^n$.
\item $f_k \to f$ locally uniformly in $\Omega$.
\end{itemize}

Then $Iu = f$ in $\Omega$.
\end{cor}

\begin{remark}
\emph{$\Gamma$-convergence} was introduced by De Giorgi in the framework of variational analysis to study convergence of sequences of functionals in Banach spaces. Here we are using the same notion of convergence for functions in $\R^n$. This type of limit usually appears in viscosity solution theory in one form or another, even though the term $\Gamma$-convergence is rarely used.
\end{remark}

\section{Comparison principle}

The comparison principle for viscosity solutions that we present here follows very standard ideas in the subject. It originated from the idea of Jensen \cite{J} of sup and inf-convolutions. The method has been succesfully applied to integro-differential equations already \cite{Sayah}. In \cite{BI} a very general proof was given where the solutions are allowed to have an arbitrary growth at infinity. Our definitions do not quite fit into the previous framework mainly because we consider the general class of operators given by Definition \ref{d:axiomatic} and we allow discontinuities outside of the domain of the equation $\Omega$. However, with small modifications, the same techniques can be adapted to our equations. We sketch the important ideas to prove the comparison principle in this section. There are two things that make the proof simpler than usual and are worth to be pointed out. One is the fact that in this paper we are only considering translation invariant equations. The other is that our operators are purely integro-differential (like \eqref{e:integral} instead of \eqref{e:fulloperator}) and they are well defined each time the functions are punctually $C^{1,1}$, which is very convenient to simplify the proof of Lemma \ref{l:Sclass}.

The result of this section that is important for the regularity theory is Theorem \ref{t:Sclass}, since we are going to apply it in section \ref{s:c1a} to incremental quotients of a solution to an equation.

In order to have a comparison principle for a nonlinear operator $I$, we need to impose a minimal ellipticity condition to our collection of linear operators $\LI$. The following assumption will suffice.
\begin{ass} \label{a:e1}
There is a constant $R_0 \geq 1$ so that for every $R>R_0$, there exists a $\delta>0$ (that could depend on $R$) such that for any operator $L$ in $\LI$, we have that $L \varphi > \delta$ in $B_R$, where $\varphi$ is given by
\[ \varphi(x) = \min( R^3 , |x|^2) \]
\end{ass}

In later sections we will need stronger assumptions to prove further regularity properties of the solutions. But for the comparison principle Assumption \ref{a:e1} is enough. Note that assumption \ref{a:e1} is very mild. It just says that given the particular function $\min( R^3 , |x|^2)$, the value of the operator will be strictly positive in $B_R$, but it does not require any unifom estimate on how that happens. If the operators $L \in \LI$ are scale invariant, it justs means that when we apply them to $\min(1,|x|^2)$ they are strictly positive in some neighborhood of the origin.

\begin{thm} \label{t:comparison}
Let $\LI$ be some class satisfying Assumption \ref{a:e1}. Let $I$ be elliptic respect to $\LI$ in the sense of definition \ref{d:axiomatic}. Let $\Omega$ be a bounded open set, and $u$ and $v$ be two functions such that 
\begin{itemize}
\item $u,v$ are bounded in $\R^n$.
\item $u$ is lower-semicontinuous at every point in $\overline \Omega$.
\item $v$ is upper-semicontinuous at every point in $\overline \Omega$.
\item $Iu \geq f$ and $Iv \leq f$ in $\Omega$.
\item $u \leq v$ in $\R^n \setminus \Omega$.
\end{itemize}
Then $u \leq v$ in $\Omega$.
\end{thm}

By $u$ being lower-semicontinuous at every point in $\overline \Omega$, we mean that $u$ is semicontinuous in $\overline \Omega$ with respect to $\R^n$. The same applies for the function $v$.

We will use the usual idea of sup- and inf-convolutions in order to prove comparison. We start by defining these concepts
\begin{definition}
Given an upper semicontinuous function $u$, the sup-convolution approximation $u^\eps$ is given by
\begin{equation}
u^\eps (x) = \sup_y u(x+y) - \frac{|y|^2}{\eps}
\end{equation}
On the other hand, if $u$ is lower semicontinuous, the inf-convolution $u_\eps$ is given by
\begin{equation}
u_\eps (x) = \inf_y u(x+y) + \frac{|y|^2}{\eps}
\end{equation}
\end{definition}

Notice that $u^\eps \geq u$ and $u_\eps \leq u$. Note also that $u^\eps$ is a supremum of translations of $u$ and $u_\eps$ is an infimum of translations of $u$.

The following two propositions are very standard, so we skip their proofs
\begin{prop} \label{p:gammaforinfconv}
If $u$ is bounded and lower-semicontinuous in $\R^n$ then $u_\eps$ $\Gamma$-converges to $u$.

If $u$ is bounded and upper-semicontinuous in $\R^n$ then $-u^\eps$ $\Gamma$-converges to $-u$.
\end{prop}

\begin{prop} \label{p:supconvE}
If $Iu \geq f$ then $I u^\eps \geq f - d_\eps$. And if $Iv \leq f$ then $I v_\eps \leq f + d_\eps$, where $d_\eps \to 0$ as $\eps \to 0$ and depends on the modulus of continuity of $f$.
\end{prop}

\begin{remark}
Proposition \ref{p:gammaforinfconv} is a straightforward generalization of the fact that $u^\eps \to u$ locally uniformly if $u$ is continuous.
\end{remark}

\begin{lemma} \label{l:uc11}
Let $u : \R^n \to \R$ be a lower semicontinuous function in $\R^n$ such that $I u \leq 0$ in $\Omega$ in the viscosity sense. Let $x$ be a point in $\Omega$ so that $u \in C^{1,1}(x)$. Then $Iu(x)$ is defined in the classical sense and $Iu(x) \leq 0$.
\end{lemma}

\begin{proof}
Use $u$ as a test function for itself with Lemma \ref{l:tc11}.
\end{proof}

\begin{lemma} \label{l:Sclass}
Let $I$ be elliptic in the sense of Definition \ref{d:axiomatic}. Let $u$ and $v$ be two bounded functions such that 
\begin{itemize}
\item $u$ is upper-semicontinuous and $v$ is lower-semicontinuous in $\R^n$.
\item $Iu \geq f$ and $Iv \leq g$ in the viscosity sense in $\Omega$. 
\end{itemize}
Then $\MLp(u-v) \geq f-g$ in $\Omega$ in the viscosity sense.
\end{lemma}

\begin{proof}
By Proposition \ref{p:supconvE}, we have that also $I u^\eps \geq f-d_\eps$ and $I v_\eps \leq g+d_\eps$. Moreover $-u^\eps \to -u$ and $v_\eps \to v$ in the $\Gamma$ sense. By the stability of viscosity solutions under $\Gamma$ limits and since $d_\eps \to 0$, it is enough to show that $\MLp (u^\eps - v_\eps) \geq f - g - 2 d_\eps$ in $\Omega$ for every $\eps > 0$.

Let $\varphi$ be a $C^2$ function touching $(u^\eps - v_\eps)$ by above at the point $x$. Note that for any $\eps>0$, both functions $u^\eps$ and $- v_\eps$ are semiconvex, which means that for each of them there is a paraboloid touching it from below at every point $x$. If a $C^2$ function touches $(u^\eps - v_\eps)$ by above at the point $x$, then both $u^\eps$ and $-v_\eps$ must be $C^{1,1}(x)$. But by Lemma \ref{l:tc11} and Lemma \ref{l:puccicontinuous}, this means that we can evaluate $Iu^\eps(x)$ and $Iv_\eps(x)$ in the classical sense and
\[ \MLp (u^\eps - v_\eps)(x) \geq I u^\eps (x) - I v_\eps (x) \geq f - g - 2 d_\eps \]
which clearly implies that also $\MLp \varphi(x) \geq f - g - 2 d_\eps$ since $\varphi$ touches $u^\eps - v_\eps$ by above. Thus $\MLp (u^\eps - v_\eps) \geq f-g - 2d_\eps$ in $\Omega$ in the viscosity sense.

Taking $\eps \to 0$ and using Lemma \ref{l:stability} we finish the proof.
\end{proof}

The result of Lemma \ref{l:Sclass} is almost the result we need to prove the comparison principle, except that we want to allow functions $u$ and $v$ that are discontinuous outside of the domain $\Omega$. We fix this last detail in the following theorem.

\begin{thm} \label{t:Sclass}
Let $I$ be elliptic in the sense of Definition \ref{d:axiomatic}. Let $u$ and $v$ be two bounded functions in $\R^n$ such that 
\begin{itemize}
\item $u$ is upper-semicontinuous and $v$ is lower-semicontinuous in $\overline \Omega$ 
\item $Iu \geq f$ and $Iv \leq g$ in the viscosity sense in $\Omega$. 
\end{itemize}
Then $\MLp(u-v) \geq f-g$ in $\Omega$ in the viscosity sense.
\end{thm}

\begin{proof}
First we will show that there exist two sequences $u_k$ and $v_k$, lower and upper semicontinuous respectively, such that
\begin{itemize}
\item $u_k = u$ in $\overline \Omega$ for every $n$.
\item $v_k = v$ in $\overline \Omega$ for every $n$.
\item $u_k \to u$ and $v_k \to v$ a.e. in $\R^n \setminus \overline \Omega$.
\item $I u_k \geq f_k$ and $I v_k \leq g_k$ with $f_k \to f$ and $g_k \to g$ locally uniformly in $\Omega$. 
\end{itemize}

It is clear that we can find two sequences $u_k$ and $v_k$ satisfying the first three items above by doing a standard mollification of $u$ and $v$ away from $\Omega$ and then \emph{filling the gap} in a semicontinuous way. What we will show is that then there are functions $f_k$ and $g_k$ for which the fourth item also holds.

The function $u_k-u$ vanishes in $\Omega$ and thus $\MLm (u_k - u)$ is defined in the classical sense in $\Omega$. Moreover
\[ \Mm_{\LI} (u_k - u) (x) \geq -2 \int_{\R^n \setminus B_{\dist(x,\bdary \Omega)}(x)} |u_k(x+y) - u(x+y)| K(y) \dd y =: h_k(x)\]

Note that $h_k$ is continuous in $\Omega$ and by dominated convergence $h_k \to 0$ locally uniformly in $\Omega$ as $k \to \infty$.

Let $\varphi$ be function touching globally $u_k$ by above at a point $x$, assuming only that $\varphi \in C^{1,1}(x)$. Then also $\varphi + u - u_k \in C^{1,1}(x)$. But $\varphi + u - u_k$ touches $u$ from above at $x$, so by Lemma \ref{l:tc11} $I(\varphi + u - u_k)(x) \geq f(x)$. But now
\[ I\varphi(x) \geq I(\varphi + u - u_k)(x) + \MLm(u - u_k)(x) \geq f(x) + h_k(x) \]
so we prove the fourth item above for $u_k$ by choosing $f_k = f + h_k$. Similarly we prove it for $v_k$.

Now that we have such sequences $u_k$ and $v_k$ we apply Lemma \ref{l:stability} and finish the proof.
\end{proof}

\begin{lemma} \label{l:maximumprinciple}
Let $u$ be a bounded function, upper-semicontinuous at every point in $\overline \Omega$, such that $\MLp u \geq 0$ in the viscosity sense in $\Omega$. Then $\sup_{\Omega} u \leq \sup_{\R^n \setminus \Omega} u$. 
\end{lemma}

\begin{proof}
Let us choose $R>R_0$ large enough so that $\Omega \subset B_R$. For any $\eps>0$, let $\varphi_M$ be the function
\[ \varphi_M(x) = M + \eps \left( 1 - \min(R^3,|x|^2) \right) \ . \]

Note that $M \leq \varphi_M(x) \leq M+\eps$ for every $x \in \R^n$. Also, by Assumption \ref{a:e1}, there is a $\delta>0$ such that $\MLp \varphi_M(x) \leq -\eps \delta $ for any $x \in B_R$.

Let $M_0$ be the smallest value of $M$ for which $\varphi_M \geq u$ in $\R^n$. We will show that $M_0 \leq \sup_{\R^n \setminus \Omega} u$. Otherwise, if $M_0 > \sup_{\R^n \setminus \Omega} u$, there must be a point $x_0 \in \Omega$ for which $u(x_0) = \varphi_{M_0} u(x_0)$. But in that case $\varphi_{M_0}$ would touch $u$ by above at $x_0 \in \Omega$ and by the definition of $\MLp u \geq 0$ in the viscosity sense we would have that $\MLp \varphi_{M_0} \geq 0$ arriving to a contradiction. Therefore, for every $x \in \R^n$, we have
\begin{align*}
u(x) &\leq \varphi_{M_0} (x) \\
&\leq M_0 + \eps \\
&\leq \sup_{\R^n \setminus \Omega} u + \eps
\end{align*}

We finish the proof by making $\eps \to 0$.
\end{proof}

\begin{proof}[Proof of Theorem \ref{t:comparison}]
By theorem \ref{t:Sclass}, $\MLp(u-v) \geq 0$ in $\Omega$. Then Lemma \ref{l:maximumprinciple} says that $\sup_{\Omega} (u-v) \leq \sup_{\R^n \setminus \Omega} (u-v)$, which finishes the proof.
\end{proof}

Once we have the comparison principle for semicontinuous sub and supersolutions, existence of the solution of the Dirichlet problem follows using the Perron's method \cite{I} as long as we can construct suitable barriers.

\section{Second order elliptic equations}
\label{s:secondorder}
It is well known that 
\[ \lim_{\sigma \to 2}  \int_{\R^n} \frac{c_n (2-\sigma)}{|y|^{n+\sigma}} \si(u,x,y) \dd y = \lim_{\sigma \to 2} -(-\lap)^{\sigma/2}u(x) = \lap u(x) \] 

With a simple change of variables $z=Ay$, we arrive to the following identity
\begin{equation} \label{e:i1}
 \lim_{\sigma \to 2} \int_{\R^n} \frac{c_n (2-\sigma)}{\det A |A^{-1} z|^{n+\sigma}} \si(u,x,z) \dd z = \sum a_{ij} u_{ij}(x)
\end{equation}
where $\{a_{ij}\}$ are the entries of $AA^t$. 

This means that we can recover any linear second order elliptic operator as a limit of integro-differential ones like \eqref{e:i1}. Moreover let us say we have a fully nonlinear operator of the form $F(D^2 u)$. Let us assume the function $F$ is Lipschitz and monotone in the space of symmetric matrices. Then $F$ can be written as
\[ F(M) = \inf_\alpha \sup_\beta \sum a_{ij}^{\alpha \beta} M_{ij} \]
for some collection of positive matrices $\{a_{ij}^{\alpha \beta}\} = A_{\alpha\beta}A_{\alpha\beta}^t$. Thus any elliptic fully nonlinear operator can be recovered as a limit of integro-differential operators as
\[F(D^2 u) = \lim_{\sigma \to 2} \left( \inf_\alpha \sup_\beta \int \frac{c_n (2-\sigma)}{\det A_{\alpha\beta} |A_{\alpha\beta}^{-1} z|^{n+\sigma}} \si(u,x,z) \dd z \right)\]
as long as the limit commutes with the operations of infimum and supremum. That is going to be the case every time the convergence is uniform in $\alpha$ and $\beta$ which is the case for example if the matrices $A_{\alpha\beta}$ are uniformly elliptic.

Another posibility is to take a family $A_{\alpha \beta}$ so that
\[ F(D^2 u) = \lim_{\sigma \to 2} \left( \inf_\alpha \sup_\beta \int \frac{\si(u,x,A_{\alpha \beta} y)}{|y|^{n+\sigma}} \dd y \right) \ . \]

Note that we can also consider operators of the form \[ I u(x) := (2-\sigma) \int \frac{1}{|y|^{n+\sigma-2}} G \left( \frac{\si(u,x,y)}{|y|^2} , y  \right) \dd y \]
with $G(d,y)$ being an arbitrary function, lipschitz and monotone in $d$, such that $G(0,y)=0$. This suggests an unusual family of second order nonlinear equations: for $P$ a quadratic polynomial
\[ F(D^2 P) = \int_{S^1} G(P(\sigma),\sigma) \dd \sigma \ . \]

\section{A nonlocal ABP estimate.}

\label{s:abp}

The Alexandroff-Backelman-Pucci (ABP) estimate is a key ingredient in the proof of Harnack inequality by Krylov and Sofonov. It is the relation that allows us to pass from an estimate in measure, to a pointwise estimate. In this section we obtain an estimate for integro-differential equations that converges to the ABP estimate as $\sigma$ approaches $2$. In a later section, we will use this nonlocal version of the ABP theorem to prove the Harnack inequality for $\sigma$ close to $2$.

In this and the next few sections we will consider the class $\LI_0$ defined by the condition \ref{e:uniformellipticity}. We write $\Mp$ and $\Mm$ to denote $\Mp_{\LI_0}$ and $\Mm_{\LI_0}$.

Let $u$ be a function that is not positive outside the ball $B_1$. Consider its concave envelope $\Gamma$ in $B_3$ defined as
\[
\Gamma(x) := \begin{cases}
              \min \set{p(x) : \text{for all planes } p>u \text{ in } B_2} & \text{in } B_3 \\
              0 & \text{in } \R^n \setminus B_3
             \end{cases}
\]

\begin{lemma} \label{l:abp2}
Let $u \leq 0$ in $\R^n \setminus B_1$. Let $\Gamma$ be its concave envelope in $B_3$. Assume $\Mp u(x) \geq -f(x)$ in $B_1$. Let $\ro = 1/(8\sqrt{n})$, $r_k = \ro 2^{-\frac{1}{2-\sigma}-k}$ and $R_k(x) =  B_{r_k}(x) \setminus B_{r_{k+1}}(x)$.

There is a constant $C_0$ depending only on $n$, $\lambda$ and $\Lambda$ (but not on $\sigma$) such that for any $x \in \{u = \Gamma\}$ and any $M>0$, there is a $k$ such that
\begin{equation} \label{e:abp1}
|R_k(x) \cap \{u(y) < u(x) + (y-x)\cdot \grad \Gamma(x) - M r_k^2 \}| \leq C_0 \frac{f(x)}{M} |R_k(x)|
\end{equation}
where $\grad \Gamma$ stands for any element of the superdifferential of $\Gamma$ at $x$, which will coincide with its gradient, and also the gradient of $u$, when these functions are differentiable.
\end{lemma}

\begin{proof}
Since $u$ can be touched by a plane from above at $x$, from Lemma \ref{l:classic2}, $\Mp u(x)$ is defined classically and we have
\[ \Mp u(x) = (2-\sigma) \int_{\R^n} \frac{\Lambda \si^+ - \lambda \si^-}{|y|^{n+\sigma}} \dx \ . \]

Recall $\si = \si(u,x,y) := u(x+y) + u(x-y) - 2 u(x)$.

Note that if both $x+y \in B_3$ and $x-y \in B_3$ then $\si(u,x,y) \leq 0$, since $u(x) = \Gamma(x) = p(x)$ for some plane $p$ that remains above $u$ in the whole ball $B_3$. Moreover, if either $x+y \notin B_3$ or $x-y \notin B_3$, then both $x+y$ and $x-y$ are not in $B_1$, so $u(x+y)\leq 0$ and $u(x-y)\leq 0$. Therefore, in any case $\si(u,x,y) \leq 0$. Thus we have
\begin{align}
 -f(x) &\leq \Mp u(x) = (2-\sigma) \int_{\R^n} \frac{- \lambda \si^-}{|y|^{n+\sigma}} \dd y  \\
&\leq (2-\sigma) \int_{B_{r_0}(x)} \frac{- \lambda \si^-}{|y|^{n+\sigma}} \dd y
\end{align}
where $r_0 = \ro 2^{-\frac{1}{2-\sigma}}$.

Splitting the integral in the rings $R_k$ and reorganizing terms we obtain
\[ f(x) \geq (2-\sigma) \lambda \sum_{k=0}^\infty \int_{R_k(x)} \frac{\si^-}{|y|^{n+\sigma}} \dx \]

Let us assume that equation \eqref{e:abp1} does not hold. We will arrive to a contradiction. We can use the oposite of \eqref{e:abp1} to estimate each integral in the terms of the previous equation.
\begin{align}
 f(x) &\geq (2-\sigma) \lambda \sum_{k=0}^\infty \int_{R_k(x)} \frac{\si^-}{|y|^{n+\sigma}} \dx \\
&\geq c (2-\sigma) \sum_{k=0}^\infty M \frac{r_k^2}{r_k^\sigma} C_0 \frac{f(x)}{M}  \\
&\geq c (2-\sigma) \frac{\ro^2}{1-2^{-(2-\sigma)}} C_0 f(x) \\
&\geq c C_0 f(x)
\end{align}
where the last inequality holds because $(2-\sigma) \frac{1}{1-2^{-(2-\sigma)}}$ remains bounded below for $\sigma \in (0,2)$. By choosing $C_0$ large enough, we obtain a contradiction.
\end{proof}

\begin{remark}
 Note that Lemma \ref{l:abp2} implies that if $\Mp u (x) \geq g(x)$ then $u(x) \neq \Gamma(x)$ at every point where $g(x)>0$.
\end{remark}

\begin{remark}
Lemma \ref{l:abp2} would hold for any particular choice of $\rho_0$ (modifying $C_0$ accordingly). The particular choice $\rho_0=1/8\sqrt{n}$ is convenient for the proofs in section 9 later in this paper.
\end{remark}

\begin{lemma}\label{l:abp1}
Let $\Gamma$ be a concave function in $B_r$. Assume that for a small $\eps$ 
\begin{equation}
| \{ y : \Gamma(y) < \Gamma(x) + (y-x) \cdot \grad \Gamma(x) - h \} \cap (B_r \setminus B_{r/2})| \leq \eps |B_r \setminus B_{r/2}|
\end{equation}
then $\Gamma(y) \geq \Gamma(x) + (y-x) \cdot \grad \Gamma(x) - h$ in the whole ball $B_{r/2}$.
\end{lemma}

\begin{proof}
Let $y \in B_{r/2}$. There are two points $y_1$, $y_2$ in $B_r \setminus B_{r/2}$ such that
\begin{enumerate}
\item $y = (y_1+y_2)/2$.
\item $|y_1-x|=|y_2-x|=\frac 34 r$.
\end{enumerate}

\begin{figure}
\begin{center}
\input{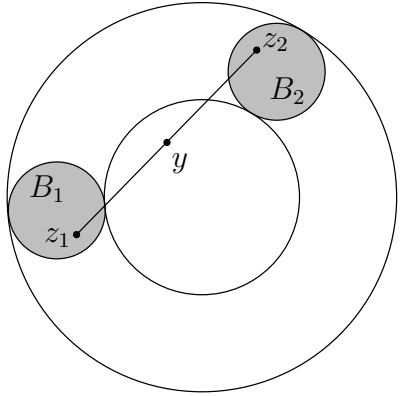}
\end{center}
\caption{The balls $B_1$ and $B_2$.}
\label{f:b1b2}
\end{figure}

Let us consider the balls $B_1 = B_{r/4}(y_1)$ and $B_2 = B_{r/4}(y_2)$ (See Figure \ref{f:b1b2}). They are symmetric respect to $y$ and they are completely contained in $B_r \setminus B_{r/2}$. If $\eps$ is small enough, there will be two points $z_1 \in B_1$ and $z_2 \in B_2$ so that
\begin{enumerate}
\item $y = (z_1+z_2)/2$
\item $\Gamma(z_1) \geq \Gamma(x) + (z_1-x) \cdot \grad \Gamma(x) - h$
\item $\Gamma(z_2) \geq \Gamma(x) + (z_2-x) \cdot \grad \Gamma(x) - h$
\end{enumerate}
and by the concavity of $\Gamma$ we finish the proof since $\Gamma(y) \geq (\Gamma(z_1)+\Gamma(z_2))/2$.
\end{proof}

\begin{cor} \label{c:abp3}
For any $\eps_0>0$ there is a constant $C$ such that for any function $u$ with the same hypothesis as in Lemma \ref{l:abp2}, there is an $r \in (0,\ro 2^{-\frac{1}{2-\sigma}})$ such that:
\begin{align}
&\frac{\abs{\set{y \in B_r \setminus B_{r/2}(x) : u(y) < u(x) + (y-x) \cdot \grad \Gamma(x) - C f(x) r^2}}}{ |B_r(x) \setminus B_{r/2}(x)|} \leq \eps_0 . \label{e:abp31}\\
&|\grad \Gamma (B_{r/4}(x))| \leq C f(x)^n |B_{r/4}(x)| \label{e:abp32}
\end{align}
Recall $\ro = 1/(8\sqrt{n})$.
\end{cor}

\begin{proof}
From Lemma \ref{l:abp2} we have \eqref{e:abp31} right away by choosing $M = C f(x)/\eps_0$. Equation \eqref{e:abp32} follows then as a consequence of Lemma \ref{l:abp1} and concavity.
\end{proof}

\begin{thm} \label{t:abp}
Let $u$ and $\Gamma$ be functions as in Lemma \ref{l:abp2}. There is a finite family of (open) cubes $Q_j$ ($j=1, \dots, m$) with diameters $d_j$ such that
\begin{enumerate}
\item[(a)] Any two cubes $Q_i$ and $Q_j$ in the family do not intersect.
\item[(b)] $\{u = \Gamma\} \subset \bigcup_{j=1}^m \overline Q_j$.
\item[(c)] $\{u = \Gamma\} \cap \overline Q_j \neq \emptyset$ for any $Q_j$.
\item[(d)] $d_j \leq \ro 2^{\frac{-1}{2-\sigma}}$, where $\ro = 1/(8\sqrt{n})$.
\item[(e)] $|\grad \Gamma(\overline Q_j)| \leq C(\max_{\overline Q_j} f )^n |Q_j|$.
\item[(f)] $|\{y \in 4 \sqrt{n} Q_j : u(y) > \Gamma(y) - C (\max_{\overline Q_j} f ) d_j^2\}| \geq \mu |Q_j|$.
\end{enumerate}

The constants $C>0$ and $\mu>0$ depend on $n$, $\Lambda$ and $\lambda$ (but not on $\sigma$).
\end{thm}

\begin{figure}[bth]
\begin{center}
\input{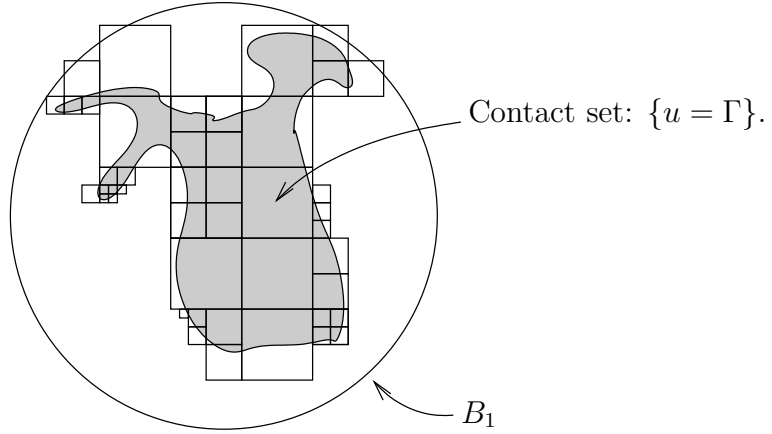}
\end{center}
\caption{The family of cubes covering $\{u = \Gamma\}$.}
\label{f:cubes}
\end{figure}

\begin{proof}
In order to obtain such family we start by covering $B_1$ with a tiling of cubes of diameter $\ro 2^{\frac{-1}{2-\sigma}}$. We discard all those that do not intersect $\{u = \Gamma\}$. Whenever a cube does not satisfy (e) and (f), we split it into $2^n$ cubes of half diameter and discard those whose closure does not intersect $\{u = \Gamma\}$. The problem is to prove that eventually all cubes satisfy (e) and (f) and this process finishes after a finite number of steps.

Let us assume the process does not finish in a finite number of steps. We assume it produces an infinite sequence nested of cubes. The intersection of their closures will be a point $x_0$. Since all of them intersect the contact set $\{u = \Gamma\}$, which is a closed set, then $u(x_0) = \Gamma(x_0)$. We will now find a contradiction by showing that eventually one of these cubes containing $x_0$ will not split.

Given $\eps_0>0$, by Corollary \ref{c:abp3}, there is a radius $r$ with $0<r<\ro 2^{\frac{-1}{2-\sigma}}$ such that
\begin{align}
&\frac{\abs{\set{y \in B_r(x_0) \setminus B_{r/2}(x_0) : u(y) < u(x_0) + (y-x_0) \cdot \grad \Gamma(x_0) - C f(x_0) r^2}}}{ |B_r(x_0) \setminus B_{r/2}(x_0)|} \leq \eps_0. \label{e:a1}\\
&|\grad \Gamma (B_{r/4}(x_0))| \leq C f(x_0)^n |B_{r/4}(x_0)| \label{e:a2}
\end{align}

There is a cube $Q_j$, with $x_0 \in \overline Q_j$, with diameter $d_j$, such that $r/4 < d_j < r/2$. Therefore (see Figure \ref{f:qandb})
\begin{align}
 B_{r/2}(x_0) &\supset \overline Q_j \\
B_{r}(x_0) &\subset 4 \sqrt{n} Q_j
\end{align}

\begin{figure}[th]
\begin{center}
\input{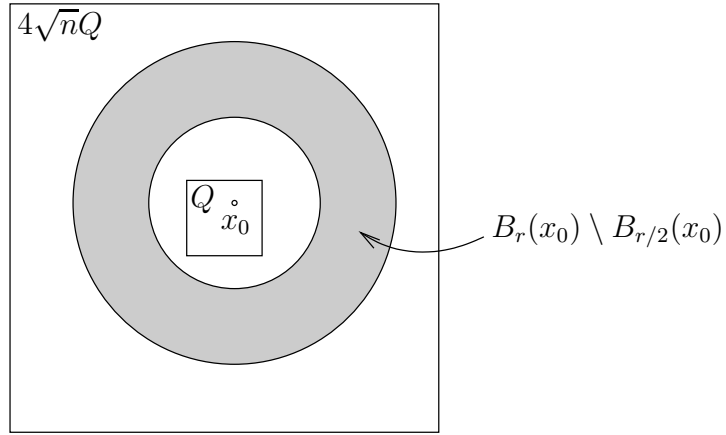}
\end{center}
\caption{The largest cube in the family containing $x_0$ and contained in $B_{r/2}$.}
\label{f:qandb}
\end{figure}

Recall that in $B_2$, $\Gamma(y) \leq u(x_0) + (y-x_0) \cdot \grad \Gamma(x_0)$ simply because $\Gamma$ is concave and $\Gamma(x_0) = u(x_0)$. Using \eqref{e:a1} and that $d_j$ and $r$ are comparable, we get
\begin{align*}
|\{ y \in 4 \sqrt{n} Q_j &: u(y) \geq \Gamma(y) - C (\max_{\overline Q_j} f) d_j^2 \} | \geq \\
&\geq \abs{\set{y \in 4\sqrt{n} Q_j : u(y) \geq u(x_0) + (y-x_0) \cdot \grad \Gamma(x_0) - C f(x_0) r^2 }} \\
&\geq (1-\eps_0) |B_r(x_0) \setminus B_{r/2}(x_0)| \geq \mu |Q_j|
\end{align*}

Thus (f) follows. Moreover, since $\overline Q_j \subset B_r$, also (e) holds for $Q_j$. Therefore $Q_j$ would not be split and the process must stop.
\end{proof}

\begin{remark}
Note that the upper bound for the diameters $\ro 2^\frac{-1}{2-\sigma}$ becomes very small as $\sigma$ is close to $2$. If we add $\sum |\grad \Gamma(Q_j)|$ and let $\sigma \to 2$, we obtain the classical Alexandroff estimate as the limit of the Riemann sums. For each $\sigma>0$ we have
\[ |\grad \Gamma(\{u = \Gamma\})| \leq \sum_j C(\max_{\overline Q_j} f^+ )^n |Q_j| \ . \]
As $\sigma \to 2$, the cube covering of $\{u=\Gamma\}$ becomes thinner and the above becomes the integral
\[ |\grad \Gamma(\{u = \Gamma\})| \leq C \int_{\{u=\Gamma\}} f^+(x)^n \dd x \ . \]
\end{remark}

\section{A special function}

\label{s:specialfunction}

In this section we only construct a special function that is a subsolution of a minimal equation outside a small ball. The importance of this function is that it is strictly positive in a larger ball and we will use that fact in a later section to prove the Harnack inequality.

\begin{lemma} \label{l:negp}
There is a $p>0$ and $\sigma_0 \in (0,2)$ such that the function
\[ f(x) = \min( 2^p , |x|^{-p} ) \]
is a subsolution to
\begin{equation} \label{e:t1}
\Mm f(x) \geq 0 
\end{equation}
for every $\sigma_0<\sigma<2$ and $|x|>1$.
\end{lemma}

\begin{proof}
It is enough to show \eqref{e:t1} for $x = e_1 = (1,0,\dots,0)$. For every other $x$ such that $|x|=1$, the relation follows by rotation. If $|x|>1$, we can consider the function $\tilde f(y) = |x|^p f(|x|y) \geq f(y)$, thus $\Mm f(x) = C \Mm \tilde f(x/|x|) \geq C \Mm f(x/|x|) > 0$.

Let $x = e_1 = (1,0,\dots,0)$. We use the following elementary relations that hold for any $a>b>0$ and $q>0$,
\begin{align}
(a+b)^{-q} &\geq a^{-q} (1 - q \frac b a) \\
(a+b)^{-q} + (a-b)^{-q} &\geq 2 a^{-q} + \frac{1}{2} q (q+1) b^2 a^{-q-2}
\end{align}
then for $|y|<1/2$,
\begin{align*}
\si &= |x+y|^{-p} + |x-y|^{-p} - 2|x|^{-p} \\
&= (1 + |y|^2 + 2 y_1)^{-p/2}+(1 + |y|^2 - 2 y_1)^{-p/2}-2 \\
&\geq 2 (1+|y|^2)^{-p/2} + \frac{1}{2}p(p+2) y_1^2 (1 + |y|^2)^{-p/2-2} - 2 \\
&\geq p \left( -|y|^2 + \frac{1}{2}(p+2) y_1^2 - \frac{1}{4}(p+2)(p+4) y_1^2 |y|^2 \right)
\end{align*}

We choose $p$ large such that
\begin{equation} \label{e:r1}
\frac{1}{2} (p+2) \lambda \int_{\bdary B_1} y_1^2 \dd \sigma(y) - \Lambda |\bdary B_1| = \delta_0 > 0
\end{equation}

We use the above relation to bound the part of the integral in the definition of $\Mm$ for which $y$ stays in a small ball $B_r$ (with $r<1/2$). We estimate $\Mm f(e_1)$.
\begin{align*}
\Mm f(e_1) &= (2-\sigma) \int_{B_r} \frac{\lambda \si^+ - \Lambda \si^-}{|y|^{n+\sigma}} \dd y + (2-\sigma) \int_{\R^n \setminus B_r} \frac{\lambda \si^+ - \Lambda \si^-}{|y|^{n+\sigma}} \dd y \\
&\geq (2-\sigma) C \int_0^r \frac{ \lambda p \delta_0 s^2 - \frac{1}{4} p(p+2)(p+4) C \Lambda s^4}{s^{1+\sigma}} \dd s - (2-\sigma) \int_{\R^n \setminus B_r} \Lambda \frac{2^p}{|y|^{n+\sigma}} \dd y \\
&\geq c r^{2-\sigma} p \delta_0 - p (p+2)(p+4) C \frac{2-\sigma}{4-\sigma} r^{4-\sigma} - \frac{2-\sigma}{\sigma} C 2^{p+1} r^{-\sigma}
\end{align*}
where we used \eqref{e:r1} to bound the first integral and that $0 \leq f(x) \leq 2^p$ to bound the second integral. Now we choose (and fix) $r \in (0,1/2)$ small, and then take $\sigma_0$ close enough to $2$, so that if $2>\sigma>\sigma_0$, the factor $(2-\sigma)$ makes the second and third terms small enough so that we get
\[ \Mm f(e_1) \geq \frac{c r^{2-\sigma} p \delta_0 }{2} > 0 \]
which finishes the proof.
\end{proof}

\begin{cor} \label{c:negp}
Given any $\sigma_0 \in (0,2)$, there is a $p>0$ and $\delta$ such that the function
\[ f(x) = \min( \delta^{-p} , |x|^{-p} ) \]
is a subsolution to
\begin{equation}
\Mm f(x) \geq 0 
\end{equation}
for every $\sigma_0<\sigma<2$ and $|x|>1$.
\end{cor}

\begin{proof}
The only difference with Lemma \ref{l:negp} is that now we are given the value of $\sigma_0$ beforehand. Let $\sigma_1$ and $p_0$ be the $\sigma_0$ and $p$ of Lemma \ref{l:negp}. So we know that for $\sigma > \sigma_1$, the result of the Corollary holds if $\delta = 1/2$ and $p=p_0$. If we take $\delta<1/2$ we are only making the function larger away from $x$, so the result will still hold for $\sigma > \sigma_1$. Now we will pick $\delta$ smaller so that the result also holds for $\sigma_0 < \sigma \leq \sigma_1$.

The key is that if $p \geq n$, $|x|^{-p}$ is not integrable around the origin. So we take $p = \max(p_0,n)$. Now, let $x = e_1$ as in the proof of lemma \ref{l:negp}. Assume $\sigma_0 < \sigma \leq \sigma_1$. We write
\begin{align*}
\Mm f(e_1) &= (2-\sigma) \int_{\R^n} \frac{\lambda \si^+}{|y|^{n+\sigma}} \dd y - (2-\sigma) \int_{\R^n} \frac{\Lambda \si^-}{|y|^{n+\sigma}} \dd y \\
&=:  I_1 + I_2
\end{align*}
where $I_1$ and $I_2$ represent the two terms in the right hand side above. Since $\sigma > \sigma_0$, $f \in C^2(x)$ and $f$ is bounded below, we have $I_2 \geq -C$ for some constant $C$ depending on $\sigma_0$, $\lambda$, $\Lambda$ and dimension. On the other hand, since $\sigma \leq \sigma_1$ and $(|x+y|^{-p}+|x-y|^{-p}-|x|^{-p})^+$ is not integrable, if we choose $\delta$ small enough we can make $I_1$ be as large as we wish. In particular, we can choose $\delta$ such that $I_1 > C > -I_2$, thus $\Mm f(e_1) > 0$. 
\end{proof}


\begin{cor} \label{c:phi}
Given any $\sigma_0 \in (0,2)$, there is a function $\Phi$ such that
\begin{itemize}
\item $\Phi$ is continuous in $\R^n$.
\item $\Phi(x) = 0$ for $x$ outside $B_{2\sqrt{n}}$.
\item $\Phi(x) > 2$ for $x \in Q_3$.
\item $\Mm \Phi > -\psi(x)$ in $\R^n$ for some positive function $\psi(x)$ supported in $\overline B_{1/4}$.
\end{itemize}
for every $\sigma>\sigma_0$.
\end{cor}

\begin{proof}
Let $p$ and $\delta$ be as in Corollary \ref{c:negp}. We consider 
\[ \Phi = c \begin{cases}
             0 &\text{in } \R^n \setminus B_{2\sqrt{n}} \\
	     |x|^{-p} - (2\sqrt{n})^{-p} &\text{in } B_{2\sqrt{n}} \setminus B_\delta \\
	     q &\text{in } B_\delta
            \end{cases} \]
where $q$ is a quadratic paraboloid chosen so that $\Phi$ is $C^{1,1}$ accross $\bdary B_\delta$. We choose the constant $c$ so that $\Phi(x) > 2$ for $x \in Q_3$ (Recall $Q_3 \subset B_{3\sqrt{n}/2} \subset B_{2\sqrt{n}}$). Since $\Phi \in C^{1,1}(B_{2\sqrt{n}})$, $\Mm \Phi$ is continuous in $B_{2\sqrt{n}}$ and from Corollary \ref{c:negp}, $\Mm \Phi \geq 0$ outside $B_{1/4}$.
\end{proof}

\section{Point estimates}

The main ingredient in the proof of Harnack inequality, as shown in \cite{CC}, is a lemma that links a pointwise estimate with an estimate in measure. The corresponding lemma in our context is the following.

\begin{lemma} \label{l:keylemma}
Let $\sigma>\sigma_0>0$. There exist constants $\eps_0>0$, $0<\mu<1$ and $M>1$ (depending only on $\sigma_0$, $\lambda$, $\Lambda$ and dimension) such that if
\begin{itemize}
\item $u \geq 0$ in $\R^n$.
\item $\inf_{Q_3} u \leq 1$.
\item $\Mm u \leq \eps_0$ in $Q_{4\sqrt{n}}$.
\end{itemize}
then 
$ |\{u \leq M \} \cap Q_1 | > \mu $.
\end{lemma}

By $Q_r(x)$ we mean the open cube $\{ y : |y_j - x_j| \leq r/2 \text{ for every j} \}$, and $Q_r := Q_r(0)$. We will also use the following notation for dilations: if $Q = Q_r(x)$, then $\lambda Q := Q_{\lambda r}(x)$. 

If we assume $\sigma \leq \sigma_1 < 2$, there is a simpler proof of Lemma \ref{l:keylemma} using the ideas from \cite{S1}. The result here is more involved because we want an estimate that remains uniform as $\sigma \to 2$.

\begin{proof}
Consider $v := \Phi - u$, where $\Phi$ is the special function constructed in Corollary \ref{c:phi}. We want to apply Theorem \ref{t:abp} (rescaled) to $v$. Note that $\Mp v \geq \Mm \Phi - \Mm (u) \geq -\psi-\eps_0$. Let $\Gamma$ be the concave envelope of $v$ in $B_{6\sqrt{n}}$.

Let $Q_j$ be the family of cubes given by Theorem \ref{t:abp}. We have
\begin{align*}
\max v &\leq C | \grad \Gamma(B_{2 \sqrt{n}}) |^{1/n} \leq \left( \sum_j |\grad \Gamma(\overline Q_j)| \right)^{1/n} \\
& \leq \left( C \sum_j ( \max_{Q_j} (\psi+\eps_0)^+ )^n |Q_j| \right)^{1/n} \\
& \leq C \eps_0 + C \left( \sum_j (\max_{Q_j} \psi^+)^n |Q_j| \right)^{1/n}
\end{align*}

However, since $\max_{Q_3} u \leq 1$ and $\min_{Q_3} \Phi \geq 2$, then $\max v \geq 1$ and we have
\[ 1 \leq C \eps_0 + C \left( \sum_j (\max_{Q_j} \psi^+)^n |Q_j| \right)^{1/n} \]

If we choose $\eps_0$ small enough, this will imply
\[ \frac{1}{2} \leq C \left( \sum_j (\max_{Q_j} \psi^+)^n |Q_j| \right)^{1/n} \]

Recall that $\psi$ is supported in $\overline B_{1/4}$ and it is bounded, thus:
\[
\frac 12 \leq C \left( \sum_{Q_j \cap B_{1/4} \neq \emptyset} |Q_j| \right)^{1/n} 
\]

Which provides a bound below for the sum of the volumes of the cubes $Q_j$ that intersect $B_{1/4}$. 
\begin{equation} \label{e:h1}
\sum_{Q_j \cap B_{1/4} \neq \emptyset} |Q_j| \geq c
\end{equation}

The diameters of all cubes $Q_j$ are bounded by $\ro 2^{\frac{-1}{2-\sigma}}$, which is always smaller than $\ro = 1/(8 \sqrt{n})$. Therefore, every time $Q_j$ intersects $B_{1/4}$, the cube $4\sqrt{n}Q_j$ will be contained in $B_{1/2}$.

Let $M_0 := \min_{B_{1/2}} \Phi$. By Theorem \ref{t:abp}, we have
\begin{equation} \label{e:h2}
 |\{x \in 4 \sqrt{n} Q_j : v(x) \geq \Gamma(x) - C d_j^2 \} | \geq c |Q_j|
\end{equation}
and $C d_j^2<C\rho_0^2$.

Let us consider the cubes $4 \sqrt{n} Q_j$, for every cube $Q_j$ that intersects $B_{1/4}$. It provides an open cover of the union of the corresponding cubes $\overline Q_j$ and it is contained in $B_{1/2}$. We take a subcover with finite overlapping that also covers the union of the original $\overline Q_j$. Combining \eqref{e:h1} with \eqref{e:h2} we obtain
\[ | \{ x \in B_{1/2} : v(x) \geq \Gamma(x) - C \rho_0^2 \} | \geq c \]
Then
\[ | \{ x \in B_{1/2} : u(x) \leq M_0 + C \rho_0^2 \} | \geq c \]

Let $M = M_0 + C \rho_0^2$. Since $B_{1/2} \subset Q_1$, we have
\[ | \{ x \in Q_1 : u(x) \leq M \} | \geq c \]
which finishes the proof.
\end{proof}

Lemma \ref{l:keylemma} is the key to the proof of Harnack inequality. The following Lemma is a consequence of Lemma \ref{l:keylemma} as it is shown in Lemma 4.6 in \cite{CC}. We have intentionally written Lemma \ref{l:keylemma} and the following one identical to their corresponding versions in \cite{CC}.

\begin{lemma}
Let $u$ be as in lemma \ref{l:keylemma}. Then
\[ | \{ u > M^k \} \cap Q_1 | \leq (1-\mu)^k \]
for $k = 1,2,3,\dots$, where $M$ and $\mu$ are as in Lemma \ref{l:keylemma}.

As a consequence, we have that
\[ |\{u \geq t\} \cap Q_1| \leq d t^{-\eps} \qquad \forall t > 0 \]
where $d$ and $\eps$ are positive universal constants.
\end{lemma}

By a standard covering argument we obtain the following theorem.

\begin{thm} \label{t:wharnack0}
Let $u \geq 0$ in $\R^n$, $u(0) \leq 1$, and $\Mm u \leq \eps_0$ in $B_{2}$ (supersolution). Assume $\sigma \geq \sigma_0$ for some $\sigma_0>0$. Then
\[ |\{ u > t \} \cap B_1 | \leq C t^{-\eps} \qquad \text{for every $t>0$.} \]
where the constant $C$ depends on $\lambda$, $\Lambda$, $n$ and $\sigma_0$.
\end{thm}

Scaling the above theorem we obtain the following version.
\begin{thm} \label{t:wharnack}
Let $u \geq 0$ in $\R^n$ and $\Mm u \leq C_0$ in $B_{2r}$ (supersolution). Assume $\sigma \geq \sigma_0$ for some $\sigma_0>0$. Then
\[ |\{ u > t \} \cap B_r | \leq C r^n (u(0)+C_0 r^\sigma)^\eps t^{-\eps} \qquad \text{for every $t$.} \]
where the constant $C$ depends on $\lambda$, $\Lambda$, $n$ and $\sigma_0$.
\end{thm}

For second order equations, Theorems \ref{t:wharnack0} and \ref{t:wharnack} are referred in the literature as $u$ being in $L^\eps$ (See \cite{CC}). 

\section{Harnack inequality}

Harnack inequality is a very important tool in analysis. In this section we obtain a version for integro-differential equations. Our estimate depends only on a lower bound $\sigma \geq \sigma_0>0$ but it remains uniform as $\sigma \to 2$. In that respect, we can consider this estimate as a generalization of Krylov-Safonov Harnack inequality.

This section is not needed for the rest of the paper because we will prove our regularity results using Theorem \ref{t:wharnack} only. A reader interested only in the regularity results can skip this section.

\begin{thm} \label{t:harnack}
Let $u \geq 0$ in $\R^n$, $\Mm u \leq C_0$ and $\Mp u \geq -C_0$ in $B_2$. Assume $\sigma \geq \sigma_0$ for some $\sigma_0>0$. Then $u(x) \leq C (u(0)+C_0)$ for every $x \in B_{1/2}$.
\end{thm}

\begin{proof}
Dividing by $u(0)+C_0$, it is enough to consider $u(0) \leq 1$ and $C_0=1$.

Let $\eps>0$ be the one from Theorem \ref{t:wharnack}. Let $\gamma = n/\eps$. Let us consider the minimum value of $t$ such that
\[ u(x) \leq h_t(x) := t (1-|x|)^{-\gamma} \text { for every } x \in B_1. \]
There must be an $x_0 \in B_1$ such that $u(x_0) = h_t(x_0)$, otherwise we could make $t$ smaller. Let $d = (1-|x_0|)$ be the distance from $x_0$ to $\bdary B_1$.

For $r=d/2$, we want to estimate the portion of the ball $B_r(x_0)$ covered by $\{u < u(x_0)/2\}$ and by $\{u > u(x_0)/2\}$. We will show that $t$ cannot be too large. In this way we obtain the result of the theorem, since the upper bound $t<C$ implies that $u(x) < C (1-|x|)^{-\gamma}$.

Let us first consider $A := \{u > u(x_0)/2\}$. By the $L^\eps$ estimate (Theorem \ref{t:wharnack0}) we have
\begin{align*}
|A \cap B_1| &\leq C \abs{\frac{2}{u(x_0)}}^{\eps} \\
&\leq C t^{-\eps} d^n
\end{align*}

Whereas $|B_r| = C d^n$, so if $t$ is large, $A$ can cover only a small portion of $B_r(x_0)$ at most.
\begin{equation} \label{e:a0}
\abs{ \{u>u(x_0)/2\} \cap B_r(x_0) } \leq C t^{-\eps} \abs{B_r}
\end{equation}

In order to get a contradiction, we will show that $\abs{\{u<u(x_0)/2\} \cap B_r(x_0)} \leq (1-\delta) B_r$ for a positive constant $\delta$ independent of $t$. 

We estimate $\abs{\{u<u(x_0)/2\} \cap B_{\theta r} (x_0)}$ for $\theta>0$ small. For every $x \in B_{\theta r} (x_0)$ we have $u(x) \leq h_t(x) \leq (d-\theta d/2)^{-\gamma} \leq u(x_0) (1-\theta/2)^{-\gamma}$, with $(1-\theta/2)^{-\gamma}$ close to one.

Let us consider 
\[ v(x) = (1-\theta/2)^{-\gamma} u(x_0) - u(x) \]
so that $v \geq 0$ in $B_{\theta r}$, and also $\Mm v \leq 1$ since $\Mp u \geq -1$. We would want to apply Theorem \ref{t:wharnack} to $v$. The only problem is that $v$ is not positive in the whole domain but only on $B_{\theta r}$. In order to apply Theorem \ref{t:wharnack} we have to consider $w=v^+$ instead, and estimate the change in the right hand side due to the truncation error.

We want to find an upper bound for $\Mm w = \Mm v^+$ instead of $\Mm v$. We know that
\[
\Mm v(x) = (2-\sigma) \int_{\R^n} \frac{\lambda \si(v,x,y)^+ - \Lambda \si(v,x,y)^-}{|y|^{n+\sigma}} \dx 
\leq 1.
\]

Therefore
\begin{align}
\Mm w &= (2-\sigma) \int_{\R^n} \frac{\lambda \si(w,x,y)^+ - \Lambda \si(w,x,y)^-}{|y|^{n+\sigma}} \dx \\
& \leq 1 + (2-\sigma) \int_{\R^n \cap \{v(x+y)<0\}} -\Lambda \frac{v(x+y)}{|y|^{n+\sigma}} \dx \\
& \leq 1 + (2-\sigma) \int_{\R^n \setminus B_{\theta r}(x_0)} \Lambda \frac{(u(x+y)-(1-\theta/2)^{-\gamma} u(x_0))^+}{|y|^{n+\sigma}} \dx \label{e:aa1}
\end{align}

Notice that the restriction $u \geq 0$ does not provide an upper bound for this last expression. We must obtain it in a different way.

Let us consider the largest value $\tau>0$ such that $u(x) \geq g_\tau := \tau (1-|4x|^2)$. There must be a point $x_1 \in B_{1/4}$ such that $u(x_1) = \tau (1-|4x_1|^2)$. The value of $\tau$ cannot be larger than $1$ since $u(0) \leq 1$. Thus we have the upper bound
\begin{align}
(2-\sigma) &\int_{\R^n} \frac{\si(u,x_1,y)^-}{|y|^{n+\sigma}} \dx \\
&\leq (2-\sigma) \int_{\R^n} \frac{\si(g_\tau,x_1,y)^-}{|y|^{n+\sigma}} \dx \leq C
\end{align}
for a constant $C$ that is independent of $\sigma$.

Since $\Mm u(x_1) \leq 1$, then
\[ (2-\sigma) \int_{\R^n} \frac{\si(u,x_1,y)^+}{|y|^{n+\sigma}} \dx \leq C \ .\]

In particular since $u(x_1) \leq 1$ and $u(x_1-y) \geq 0$,
\[ (2-\sigma) \int_{\R^n} \frac{(u(x_1+y)-2)^+}{|y|^{n+\sigma}} \dx \leq C \ .\]

We can use the inequality above to estimate \eqref{e:aa1}. We can assume $u(x_0) > 2$, since otherwise $t$ would not be large.

\begin{align*}
(2-\sigma) &\int_{\R^n \setminus B_{\theta r}(x_0)} \Lambda \frac{(u(x+y)-(1-\theta/2)^{-\gamma} u(x_0))^+}{|y|^{n+\sigma}} \dx \\
&\leq (2-\sigma) \int_{\R^n \setminus B_{\theta r}(x_0)} \Lambda \frac{(u(x_1+y+x-x_1)-(1-\theta/2)^{-\gamma} u(x_0))^+}{|y+x-x_1|^{n+\sigma}} \ \frac{|y+x-x_1|^{n+\sigma}}{|y|^{n+\sigma}} \dd y \\
&\leq C (\theta r)^{-n-\sigma}
\end{align*}

So finally we obtain
\[
\Mm w \leq C (\theta r)^{-n-\sigma}
\]

Now we can apply Theorem \ref{t:wharnack} to $w$ in $B_{\theta r}$. Recall $w(x_0) =((1-\theta/2)^{-\gamma}-1)u(x_0)$, we have
\begin{align*}
\left\vert \set{ u < \frac{u(x_0)}{2} } \cap B_{\frac{\theta r}{2}} \right\vert &= | \{ w > u(x_0) ((1-\theta/2)^{-\gamma}-1/2) \} \cap B_{\theta r /2} | \\
\leq C (\theta r)^n &\left( ((1-\theta/2)^{-\gamma}-1 )u(x_0)+C (\theta r)^{-n-\sigma} (r\theta)^\sigma \right)^\eps \left( u(x_0) ((1-\theta/2)^{-\gamma}- \frac 1 2 ) \right)^{-\eps} \\
\leq C (\theta r)^n &\left( ((1-\theta/2)^{-\gamma}-1)^\eps + \theta^{-n \eps} t^{-\eps} \right)
\end{align*}

Now let us choose $\theta>0$ so that the first term is small:
\[ C (\theta r)^n ((1-\theta/2)^{-\gamma}-1)^\eps \leq \frac{1}{4} \abs{B_{\theta r/2}} \ . \]

Notice that the choice of $\theta$ is independent of $t$. For this fixed value of $\theta$ we observe that if $t$ is large enough, we will also have
\[ C (\theta r)^n \theta^{-n \eps} t^{-\eps} \leq \frac{1}{4} \abs{B_{\theta r/2}} \]
and therefore
\[ | \{ u < u(x_0)/2 \} \cap B_{\theta r /2} | \leq \frac{1}{2} \abs{B_{\theta r/2}} \]
which implies that for $t$ large
\[ | \{ u > u(x_0)/2 \} \cap B_{\theta r /2} | \geq c \abs{B_r} \ . \]
But this contradicts \eqref{e:a0}. Therefore $t$ cannot be large and we finish the proof.
\end{proof}

\section{H\"older estimates.}

The purpose of this section is to prove the following H\"older regularity result.

\begin{thm} \label{t:ca}
Let $\sigma>\sigma_0$ for some $\sigma_0>0$. Let $u$ be bounded function in $\R^n$, such that
\begin{align*}
\Mp u &\geq -C_0 \qquad \text{in } B_1 \\
\Mm u &\leq C_0 \qquad \text{in } B_1
\end{align*}
then there is an $\alpha > 0$ (depending only on $\lambda$, $\Lambda$, $n$ and $\sigma_0$) such that $u \in C^\alpha(B_{1/2})$ and
\[ u_{C^\alpha(B_{1/2})} \leq C \big( \sup_{\R^n} |u| + C_0 \big) \]
for some constant $C>0$.
\end{thm}

Even though this result could be obtained as a consequence of the Harnack inequality, we will prove it using only Theorem \ref{t:wharnack}. We do it in this way because it looks potentially simpler to generalize since we proved the Harnack inequality (Theorem \ref{t:harnack}) using Theorem \ref{t:wharnack}.

Theorem \ref{t:ca} follows from the following Lemma by a simple scaling.

\begin{lemma}
Let $\sigma>\sigma_0$ for some $\sigma_0>0$. Let $u$ be a function such that
\begin{align*}
-1/2 &\leq u \leq 1/2 && \text{in } \R^n \\
\Mp u &\geq -\eps_0 && \text{in } B_1 \\
\Mm u &\leq \eps_0 && \text{in } B_1
\end{align*}
then there is an $\alpha > 0$ (depending only on $\lambda$, $\Lambda$, $n$ and $\sigma_0$) such that $u \in C^\alpha$ at the origin. More precisely
\[ |u(x) - u(0)| \leq C |x|^\alpha \]
for some constant $C$.
\end{lemma}

\begin{proof}
We will show that there exists sequences $m_k$ and $M_k$ such that $m_k \leq u \leq M_k$ in $B_{4^{-k}}$ and 
\begin{equation} \label{e:al}
M_k - m_k = 4^{-\alpha k} 
\end{equation}
so that the theorem holds with $C= 4^\alpha$.

For $k=0$ we choose $m_0=-1/2$ and $M_0=1/2$. By assumption we have $m_0 \leq u \leq M_0$ in the whole space $\R^n$. We want to construct the sequences $M_k$ and $m_k$ by induction.

Assume we have the sequences up to $m_k$ and $M_k$. We want to show we can continue the sequences by finding $m_{k+1}$ and $M_{k+1}$. 

In the ball $B_{4^{-k-1}}$, either $u \geq (M_k + m_k)/2$ in at least half of the points (in measure), or $u \leq (M_k + m_k)/2$ in at least half of the points. Let us say that $\abs{ \{u \geq (M_k + m_k)/2 \} \cap B_{4^{-k-1}} } \geq \abs{B_{4^{-k-1}}}/2$.

Consider 
\[ v(x) := \frac{u(4^{-k} x) - m_k}{(M_k-m_k)/2} \]
so that $v(x) \geq 0$ in $B_1$ and $\abs{ \{v \geq 1 \} \cap B_{1/4} } \geq \abs{B_{1/4}}/2$. Moreover, since $\Mm u \leq \eps_0$ in $B_1$,
\[ \Mm v \leq \frac{4^{-k \sigma} \eps_0}{(M_k-m_k)/2} = \eps_0 4^{k (\sigma-\alpha)} \leq \eps_0 \] if $\alpha$ is chosen less than $\sigma$.

From the inductive hypothesis, for any $j \geq 1$, we have
\[ v \geq \frac{(m_{k-j}-m_k)}{(M_k-m_k)/2} \geq \frac{(m_{k-j}- M_{k-j} + M_k - m_k)}{(M_k-m_k)/2} \geq -2 \cdot 4^{\alpha j}+2 \geq 2 (1 - 4^{\alpha j}) \qquad \text{in } B_{2^j} \]

Therefore $v(x) \geq -2 (|4x|^\alpha-1)$ outside $B_1$. If we let $w(x) = \max(v,0)$, then $\Mm w \leq \Mm v + \eps_0$ in $B_{3/4}$ if $\alpha$ is small enough. We still have $\abs{ \{w \geq 1 \} \cap B_1 } \geq \abs{B_1}/2$. Given any point $x \in B_{1/4}$, can can apply Theorem \ref{t:wharnack} in $B_{1/2}(x)$ to obtain
\[ C(w(x) + 2 \eps_0)^\eps \geq |\{w>1\} \cap B_{1/2}(x)| \geq \frac{1}{2} |B_{1/4}| \ . \]

If we have chosen $\eps_0$ small, this implies that $w \geq \theta$ in $B_{1/4}$ for some $\theta>0$. Thus if we let $M_{k+1} = M_k$ and $m_{k+1} = m_k + \theta (M_k-m_k)/2$ we have $m_{k+1} \leq u \leq M_{k+1}$ in 
$B_{2^{k+1}}$. Moreover $M_{k+1} - m_{k+1} = (1 - \theta/2) 2^{-\alpha k}$. So we must choose $\alpha$ and $\theta$ small and so that $(1 - \theta/2) = 4^{-\alpha}$ and we obtain $M_{k+1} - m_{k+1} = 4^{-\alpha (k+1)}$

On the other hand, if $\abs{ \{u \leq (M_k + m_k)/2 \} \cap B_{4^{-k}} } \geq \abs{B_{4^{-k}}}/2$, we define
\[ v(x) := \frac{M_k - u(4^{-k} x)}{(M_k-m_k)/2} \]
and continue in the same way using that $\Mp u \geq -\eps_0$.
\end{proof}

\section{$C^{1+\alpha}$ estimates.}
\label{s:c1a}
In this section we prove an interior $C^{1,\alpha}$ regularity result for the solutions to a general class of fully nonlinear integro-differential equations. The idea of the proof is to use the H\"older estimates of Theorem \ref{t:ca} to incremental quotients of the solution. There is a difficulty in that we have no uniform bound in $L^\infty$ for the incremental quotients outside of the domain. This becomes an issue since we are dealing with nonlocal equations. The way we solve it is by assuming some extra regularity of the family of integral operators $\LI$. The extra assumption, compared to the assumptions for H\"older regularity \eqref{e:uniformellipticity}, is a modulus of continuity of $K$ in measure, so as to make sure that far away oscillations tend to cancel out.

Given $\rho_0>0$, we define the class $\LI_1$ by the operators $L$ with kernels $K$ such that
\begin{align} 
(2-\sigma)\frac{\lambda}{|y|^{n+\sigma}} &\leq K(y) \leq (2-\sigma)\frac{\Lambda}{|y|^{n+\sigma}} \label{e:lic1a1}\\ 
\int_{\R^n \setminus B_{\rho_0}} \frac{|K(y)-K(y-h)|}{|h|} \dd y &\leq C \qquad \text{every time $|h|<\frac {\rho_0} 2$} \label{e:lic1a2}
 \end{align}
 
A simple condition for \eqref{e:lic1a2} to hold would be that $|\grad K(y)| \leq \frac{\Lambda}{|y|^{1+n+\sigma}}$.

In the following theorem we give interior $C^{1,\alpha}$ estimates for fully nonlinear elliptic equations.

\begin{thm} \label{t:c1a} Assume $\sigma>\sigma_0$. There is a $\rho_0>0$ (depending on $\lambda$, $\Lambda$, $\sigma_0$ and $n$) so that if $I$ is a nonlocal elliptic operator with respect to $\LI_1$ in the sense of Definition \ref{d:axiomatic} and $u$ is a bounded function such that $I u = 0$ in $B_1$,
then there is a universal (depends only on $\lambda$, $\Lambda$, $n$ and $\sigma_0$) $\alpha>0$ such that $u \in C^{1+\alpha}(B_{1/2})$ and
\[ u_{C^{1+\alpha}(B_{1/2})} \leq C\left( \sup_{\R^n} |u|+|I0| \right) \]
for some constant $C>0$ (where by $I0$ we mean the value we obtain when we apply $I$ to the function that is constant equal to zero). The constant $C$ depends on $\lambda$, $\Lambda$, $\sigma_0$, $n$ and the constant in \eqref{e:lic1a2}.
\end{thm}

\begin{proof}
Because of the assumption \eqref{e:lic1a1}, the class $\LI_1$ is included in $\LI_0$ given by \ref{e:uniformellipticity}. Since $I u = 0$ in $B_1$, in particular $\Mp u \geq Iu - I0 = -I0$ and also $\Mm u \leq I0$ in $B_1$, and therefore by Theorem \ref{t:ca} we have $u \in C^\alpha (B_{1-\delta})$ for any $\delta > 0$ with $\norm{u}_{C^\alpha} \leq C ( \sup |u|+|I0|)$.

Now we want to improve the obtained regularity iteratively by applying Theorem \ref{t:ca} again until we obtain Lipschitz regularity in a finite number of steps.

Assume we have proved that $u \in C^\beta(B_r)$ for some $\beta>0$ and $1/2 < r <1$. We want to apply Theorem \ref{t:ca} for the difference quotient
\[ w^h = \frac{ u(x+h) - u(x) }{|h|^\beta} \]
to obtain $u \in C^{\beta+\alpha}(B_{r-\delta})$. By Theorem \ref{t:Sclass}, $\Mp_{\LI_1} w^h \geq 0$ and $\Mm_{\LI_1} w^h \leq 0$ in $B_r$. In particular $\Mp w^h \geq 0$ and $\Mm w^h \leq 0$ in $B_r$.

The function $w^h$ is uniformly bounded in $B_r$ because $u \in C^\beta(B_r)$. Outside $B_r$ the function $w^h$ is not uniformly bounded, so we cannot apply Theorem \ref{t:ca} immediately. However, $w^h$ has oscillations that cause cancellations in the integrals because of our assumption \eqref{e:lic1a2}.

Let $\eta$ be a smooth cutoff function supported in $B_r$ such that $\eta \equiv 1$ in $B_{r - \delta/4}$,  where $\delta$ is some small positive number that will be determined later.

Let us write $w^h = w^h_1 + w^h_2$, where
\begin{align*}
w^h_1 &= \frac{ \eta u (x+h) - \eta u(x) }{|x|^\beta} \\
w^h_2 &= \frac{ (1-\eta) u (x+h) - (1-\eta) u(x) }{|x|^\beta} \\
\end{align*}

Let $x \in B_{r/2}$ and $|h|<\delta/16$. In this case $(1-\eta)u(x) = (1-\eta) u (x+h) = 0$ and $w^h(x) = w^h_1(x)$. We have to show that $w^h_1 \in C^{\beta+\alpha}(B_{r-\delta})$.

We have
\begin{align*}
\Mp w^h_1 \geq \Mp_{\LI_1} w^h_1 = \Mp_{\LI_1} (w^h - w^h_2) \geq 0 - \Mp_{\LI_1} w^h_2 \ , \\
\Mm w^h_1 \leq \Mm_{\LI_1} w^h_1 = \Mm_{\LI_1} (w^h - w^h_2) \leq 0 - \Mm_{\LI_1} w^h_2 \ .
\end{align*}
In order to apply Theorem \ref{t:ca}, we will show that $|\Mp_{\LI_1} w^h_2|$ and $|\Mm_{\LI_1} w^h_2|$ are bounded in $B_{r-\delta/2}$ by $C \sup|u|$ for some universal constant $C$. We must show those inequalities for any operator $L \in \LI_1$.

Since $(1-\eta)u(x) = (1-\eta) u (x+h) = 0$. $w^h(x) = w^h_1(x)$, we have the expression
\[ L w^h_2 = \int_{\R^n} \frac{ (1-\eta)u(x+y+h) - (1-\eta)u(x+y)}{|h|^\beta} K(y) \dd y \]
and we notice that both terms $(1-\eta)u(x+y+h) = (1-\eta)u(x+y) = 0$ for $|y| < \delta/8$. We take $\rho_0 = \delta/4$, therefore we can integrate by parts the incremental quotient to obtain
\begin{align*}
|L w^h_2| &= \abs{ \ \int_{\R^n} (1-\eta)u(x+y) \frac{ K(y) - K(y-h) }{|h|^\beta} \dd y} \\
&\leq \int_{\R^n} |(1-\eta)u(x+y)|  |h|^{1-\beta} \frac{ |K(y) - K(y-h)| }{|h|} \dd y  \qquad \text{using \eqref{e:lic1a2}} \\
&\leq |h|^{1-\beta} \int_{\R^n \setminus B_{\delta/4}} \frac{ |K(y) - K(y-h)| }{|h|} \dd y \ \sup_{\R^n} |u| \\
&\leq C |h|^{1-\beta} |u| \leq C \sup_{\R^n} |u|
\end{align*}

So, we have obtained $\Mp w^h_1 \geq -C \sup |u|$ and $\Mm w^h_1 \leq C \sup |u|$ in $B_{r-\delta/2}$ for $|h|<\delta/16$. We can apply theorem \ref{t:ca} to get that $w^h_1$ (and thus also $w^h$) is uniformly $C^\alpha$ in $B_{r-\delta}$. By the standard telescopic sum argument \cite{CC}, this implies that $u \in C^{\alpha+\beta}(B_{r-\delta})$.

Iterating the above argument, we obtain that $u$ is Lipschitz in $[1/\alpha]$ steps. Then, for any unit vector $e$, we use the same reasoning for the incremental quotients
\[ w^h = \frac{u(x+he) - u(x)}{h} \]
to conclude that $u \in C^{1,\alpha}$ in a smaller ball. If we choose the constant $\delta$ appropriately, we get $u \in C^{1,\alpha}(B_{1/2})$
\end{proof}

\begin{remark}
Note that the value of $\rho_0$ in Theorem \ref{t:c1a} is not scale invariant. If we want to scale the estimate to apply it to a function $u$ such that $Iu=0$ in $B_r$, then we also have to multiply the value of $\rho_0$ times $r$.
\end{remark}

\begin{remark}
Note that the family $\LI$ given by the operators $L$ with the form
\[ Lu(x) = \int_{\R^n} \frac{c_n (2-\sigma)}{\det A |A^{-1} z|^{n+\sigma}} \si(u,x,z) \dd z \]
satisfies the conditions \eqref{e:lic1a1} and \eqref{e:lic1a2}. Thus, from the arguments in section \ref{s:secondorder} and Theorem \ref{t:c1a}, we reconver the $C^{1,\alpha}$ estimates for fully nonlinear elliptic equations.
\end{remark}


\section{Truncated kernels.}
\label{s:truncated}

For applications, it is important to be able to deal with integro-differential operators whose kernels do not satisfy \eqref{e:uniformellipticity} in the whole space $\R^n$ but only in a neighborhood of the origin. For example we want to be able to deal with the operators related to truncated $\alpha$-stable Levy processes. In this section we extend our regularity resuls for this kind of operators.

We consider the following class $\LI$. We say that an operator $L$ belongs to $\LI$ if its corresponding kernel $K$ has the form
\begin{equation}
K(y) = K_1(y) + K_2(y) \geq 0 \ .
\end{equation}
Where 
\[ (2-\sigma) \frac{\lambda}{|x|^{n+\sigma}} \leq K_1(y) \leq (2-\sigma) \frac{\Lambda}{|x|^{n+\sigma}}\]
and $K_2 \in L^1(\R^n)$ with $\norm{K_2}_{L^1} \leq \kappa$.

In this class $\LI$ we can consider kernels that are comparable to $|y|^{-n-\sigma}$ near the origin but decay exponentially at infinity, or even become zero outside some ball. For example
\begin{align*}
K(y) &= \frac{1}{|y|^{n+\sigma}} e^{-|y|^2} \: \text{or} \\
K(y) &= \frac{a(y)}{|y|^{n+\sigma}} \chi_{B_1}(y) \: \text{ where } \lambda \leq a(y) \leq \Lambda \ . \\
\end{align*}

This class $\LI$ is larger than the class $\LI_0$ in \eqref{e:uniformellipticity}. However, in the following lemma we show that the extremal operators $\Mp_\LI$ and $\Mm_\LI$ are controlled by the corresponding extremal operators of $\LI_0$, $\Mp$ and $\Mm$, plus the $L^\infty$ norm of $u$.

\begin{lemma}
Let $u$ be a bounded function in $\R^n$ and $C^{1,1}$ at the point $x$. Then
\begin{align*}
\MLm u(x) &\geq \Mm u(x) - 4 \kappa \norm{u}_{L^\infty} \\
\MLp u(x) &\leq \Mp u(x) + 4 \kappa \norm{u}_{L^\infty}
\end{align*}
\end{lemma}

\begin{proof}
All we have to do is show that for each $L \in \LI$, he have $Lu(x) \geq \Mm u(x) - \kappa \inf_{\R^n} u$ and $Lu(x) \leq \Mp u(x) + \kappa \sup_{\R^n} u$.

We have
\begin{align*}
Lu &= \int \si(u,x,y) (K_1(y)+K_2(y)) \dd y \\
&= \int \si(u,x,y) K_1(y) \dd y + \int \si(u,x,y) K_2(y) \dd y \\
&\geq \Mm u(x) + \int (u(x+y) + u(x-y) - 2u(x)) K_y(y) \dd y \\
&\geq \Mm u(x) - 4 \norm{u}_{L^\infty} \norm{K_2}_{L^1}  = \Mm u(x) - 4 \kappa \norm{u}_{L^\infty}
\end{align*}

In a similar way the inequality for $\MLp u(x)$ follows.
\end{proof}

\begin{cor} \label{c:mlp}
If $u$ is bounded in $\R^n$ and in an open set $\Omega$, $\MLp u \geq -C$ and $\MLm u \leq C$, then
\begin{align*}
\Mp u &\geq -C - 4 \kappa \norm{u}_{L^\infty} \\
\Mm u &\leq C + 4 \kappa \norm{u}_{L^\infty} \ .
\end{align*}
\end{cor}

\begin{thm} \label{t:ca2}
Let $\sigma>\sigma_0$ for some $\sigma_0>0$. Let $u$ be bounded function in $\R^n$, such that
\begin{align*}
\MLp u &\geq -C_0 \qquad \text{in } B_1 \\
\MLm u &\leq C_0 \qquad \text{in } B_1
\end{align*}
then there is an $\alpha>0$ (depending only on $\lambda$, $\Lambda$, $n$ and $\sigma_0$) such that $u \in C^\alpha(B_{1/2})$ and
\[ u_{C^\alpha(B_{1/2})} \leq C \big( \norm{u}_{L^\infty} + C_0 \big) \]
for some constant $C>0$ that depends on $\lambda$, $\Lambda$, $n$ and $\sigma_0$ and $\kappa$.
\end{thm}

\begin{proof}
Form Corollary \ref{c:mlp}
\begin{align*}
\Mp u &\geq -C_0 - 4 \kappa \norm{u}_{L^\infty} \\
\Mm u &\leq C_0 + 4 \kappa \norm{u}_{L^\infty} \ .
\end{align*}

Then, from Theorem \ref{t:ca}
\begin{align*}
u_{C^\alpha(B_{1/2})} &\leq C \big( \norm{u}_{L^\infty} + C_0 + 4 \kappa \norm{u}_{L^\infty} \big) \\
&\leq \tilde C \big( \norm{u}_{L^\infty} + C_0 \big) \ .
\end{align*}
\end{proof}

If we use Theorem \ref{t:ca2} instead of Theorem \ref{t:ca} in the proof of Theorem \ref{t:c1a}, we obtain a $C^{1,\alpha}$ result for a class $\LI$ that includes kernels with exponential decay or compact support.

\begin{thm} \label{t:genc1a}
Let $\LI$ be the class of operators with kernels $K$ such that
\begin{align}
\int_{\R^n \setminus B_{\rho_0}} \frac{|K(y)-K(y-h)|}{|h|} \dd y &\leq C \qquad \text{every time $|h|<\frac {\rho_0} 2$} \label{e:prevcond}\\
K &= K_1 + K_2 \\
(2-\sigma)\frac{\lambda}{|y|^{n+\sigma}} &\leq K_1(y) \leq (2-\sigma)\frac{\Lambda}{|y|^{n+\sigma}} \\
\norm{K_2}_{L^1} &\leq \kappa
\end{align}

There is a $\rho_0>0$ so that if $I$ be a nonlocal elliptic operator in the sense of Definition \ref{d:axiomatic} and $u$ is a bounded function such that $I u = 0$ in $B_1$
then there is an $\alpha>0$ (depending only on $\lambda$, $\Lambda$, $n$ and $\sigma$) such that $u \in C^{1+\alpha}(B_{1/2})$ and
\[ u_{C^{1+\alpha}(B_{1/2})} \leq C \left( \sup_{\R^n} |u| + |I0| \right) \]
for some constant $C>0$.
\end{thm}

\begin{remark}
We can prove Theorem \ref{t:ca2} because in our $C^\alpha$ estimates we allow a bounded right hand side. Theorem \ref{t:genc1a} would be more general if the inequality \eqref{e:prevcond} was required with $K_1$ instead of $K$. In order to prove such result we would need to have $C^{1,\alpha}$ estimates like the ones of Theorem \ref{t:c1a} with a nonzero right hand side. This type of results is well known for elliptic partial differential equations \cite{C2} and we are planning to extend it to nonlocal equations in future work.

It is not hard to check that if the assumption \eqref{e:prevcond} involved $K_1$ instead of $K$, then the class $\LI$ above would be the same as the larger class $\LI_0$ of \eqref{e:uniformellipticity} and Theorem \ref{t:genc1a} would apply to a very large family of operators.
\end{remark}

\bibliographystyle{plain}   
\bibliography{nonl}             
\index{Bibliography@\emph{Bibliography}}%
\end{document}